\documentclass[12pt]{amsart}

\usepackage{amsmath, amssymb, amscd, verbatim, xspace,amsthm}
\usepackage{latexsym, epsfig, color}
\usepackage[hidelinks]{hyperref}

\usepackage[OT2,T1]{fontenc}
\DeclareSymbolFont{cyrletters}{OT2}{wncyr}{m}{n}
\DeclareMathSymbol{\Sha}{\mathalpha}{cyrletters}{"58}

\newcommand{\ba}{\begin{align*}}
\newcommand{\ea}{\end{align*}}

\newcommand{\bC}{\ensuremath{{\mathbb{C}}}}
\newcommand{\Z}{\ensuremath{{\mathbb{Z}}}\xspace}
\renewcommand{\P}{\ensuremath{{\mathbb{P}}}}

\newcommand{\Q}{\ensuremath{{\mathbb{Q}}}}
\newcommand{\R}{\ensuremath{{\mathbb{R}}}}
\newcommand{\F}{\ensuremath{{\mathbb{F}}}}

\newcommand{\E}{\ensuremath{{\mathbb{E}}}}

\newcommand{\ra}{\rightarrow}

\newcommand\Hom{\operatorname{Hom}}

\newcommand\Aut{\operatorname{Aut}}

\newcommand\Gal{\operatorname{Gal}}
\newcommand\Nm{\operatorname{Nm}}
\newcommand\Sym{\operatorname{Sym}}
\newcommand\Prob{\operatorname{Prob}}
\newcommand\Sur{\operatorname{Sur}}

\newcommand\tensor{\otimes}
\newcommand\isom{\simeq}

\newcommand\Disc{\operatorname{Disc}}
\newcommand\GL{\operatorname{GL}}

\newcommand\Spec{\operatorname{Spec}}

\newcommand\cok{\operatorname{cok}}

\renewcommand\O{\mathcal{O}}

\newcommand\Pic{\operatorname{Pic}}

\newcommand{\bCl}{\operatorname{Cl}}

\newcommand\bq{\begin{equation}}
\newcommand\eq{\end{equation}}

\numberwithin{equation}{section}

\newtheorem{theorem}[equation]{Theorem}
\newtheorem{corollary}[equation]{Corollary}

\newtheorem{lemma}[equation]{Lemma}

\theoremstyle{remark}

\usepackage{fullpage,pdfsync,bbm}

\newtheorem{nts}{Note to self}

\newtheorem{oproblem}[equation]{Open Problem}

\newcommand\rk{{\operatorname{rk}}}
\newcommand\Sel{{\operatorname{Sel}}}
\newcommand\cO{\mathcal{O}}
\newcommand\calC{\mathcal{C}}
\newcommand\calF{F}

\newcommand{\cHur}{\mathsf{Hur}}

\title{Probability theory for random groups arising in number theory}

\author{Melanie Matchett Wood}
\address{Department of Mathematics\\
Harvard University\\
Science Center Room 325\\
1 Oxford Street\\
Cambridge, MA 02138 USA}  
\email{mmwood@math.harvard.edu}
\date{November 30, 2021}

\begin{document}
\begin{abstract}
We consider the probability theory, and in particular the moment problem and universality theorems, for random groups of the sort of that arise or are conjectured to arise in number theory, and in related situations in topology and combinatorics.  The distributions of random groups that are discussed  include those conjectured in the
 Cohen-Lenstra-Martinet heuristics to be the distributions of class groups of random number fields, as well as distributions of non-abelian generalizations, and those conjectured to be the distributions of Selmer groups of random elliptic curves.  For these sorts of distributions on finite and profinite groups, we survey what is known about the moment problem and universality, give a few new results including new applications, and suggest open problems.

\end{abstract}

\maketitle

\section{Introduction}\label{S:intro}
In this paper we will discuss the probability theory of random groups that arise in number theory and related areas, and the applications of that probability theory to other fields.  We focus on the moment problem and on universality results for these random groups.
While our focus is on the probability theory, we use potential applications in number theory, as well as topology and combinatorics, to motivate the kind of random groups on which we focus our probabilistic study.

One of the first motivating examples is the Cohen-Lenstra distribution on finite abelian $p$-groups.  Let $p$ be a prime and let $X_{CL}$ be a random finite abelian $p$-group such that 
$$\Prob(X_{CL}\isom A)=\frac{\prod_{i\geq 1}(1-p^{-i})}{|\Aut(A)|}$$ for each finite abelian $p$-group $A$. 
Let $C_B$ be the Sylow $p$-subgroup of the class group of a uniform random imaginary quadratic field $K$ with $|\Disc K|\leq B$. 
Then Cohen and Lenstra \cite{Cohen1984} conjectured that for each finite abelian $p$-group $A$,
\begin{equation}
\lim_{B\ra\infty} \Prob(C_B \isom A) =\Prob(X_{CL}\isom A),
\end{equation}
i.e. that the $C_B$ converge (in distribution) to $X_{CL}$.
This $X_{CL}$ is our starting example of a random group whose probability theory we wish to understand.  Throughout the paper, we will consider more examples, including those related to generalizations of $C_B$ such as when quadratic extensions are replaced by higher degree extensions or when the base field $\Q$ is replaced by another number field or $\F_q(t)$.
We will consider non-abelian analogs where we consider $\Gal(K^{un}/K)$, the Galois group of the maximal unramified extension of $K$, in place of the class group.
We will mention connections to analogous random groups arising in other fields, such as $\pi_1(M)$ or $H_1(M)$ for a random $3$-manifold $M$, or the Jacobian (a.k.a. sandpile group) of a random graph.

With these examples in mind, we first discuss the moment problem.  Given a random variable $X$ of a certain type, based on the type of random variable, we choose certain real-valued functions $f_0,f_1,\dots$ and call the averages $\E(f_k(X))$ the \emph{moments} of $X$.  When $X$ is real valued, we usually take $f_k(X)=X^k$,
but when $X$ is a random group we usually take $f_k(X)$ to be the number $\#\Sur(X,G_k)$ of surjective homomorphisms from $X$ to a group $G_k$.  The moment problem asks when the distribution of the random variable is determined uniquely from these moments.  This is very useful in applications because the moments are usually easier to access than a distribution itself, and we will discuss many applications to class groups and their generalizations.  

Next, we discuss universality questions in the sense of the central limit theorem.  We ask, when and how can we build a random group from many independent inputs, such that in a limit, the random group is insensitive to the distribution of the random inputs?  When this happens, the output distribution is of course a natural one (as in the normal distribution in the Central Limit Theorem), and it tells us that such distributions are likely to arise in nature.  This can help provide further motivation and context for conjectures in number theory.  
As the theory develops, we expect there will be further applications of these universality results to other fields. 
 
For both topics, we will review what is known for random abelian and non-abelian groups, mention many applications, prove a few new results and applications, and suggest open problems.

\subsection{Notation and Conventions}
We use $\E$ to denote the expectation of a real-valued random  variable.

For a finite set $S$, we use $\#S$ or $|S|$ to denote the size of the set.

We write $\F_q$ for the finite field with $q$ elements.

For a set of primes $P$, a $P$-group is a group whose order is a product of powers of primes in $P$, and a pro-$P$ group is a profinite group all of whose continuous finite quotients are $P$-groups.  The pro-$P$ completion of a group is the inverse limit of all of its $P$-group quotients (and is a pro-$P$ group).

We use $\Hom$, resp. $\Sur$ to denote homomorphisms, resp. surjective homomorphisms, always in the category  of whatever  the objects are in, e.g. for profinite groups we take continuous homomorphisms, and for $R$-modules we take $R$-module homomorphisms.  Sometimes we use a subscript, e.g. $\Sur_R(A,B)$, as a reminder of the category.  We use $\Aut$ to denote automorphisms with the same caveats.

When we take a random finite group, it is always with the discrete $\sigma$-algebra on the set of finite groups.

For random variables $Y,X_0,\dots$ with respect to a Borel $\sigma$-algebra, we say the $X_n$ weakly converge in distribution to $Y$ if for every open set $U$ we have $\lim \inf \Prob(X_n\in U) \geq \Prob(Y\in U)$.  By the Portmanteau theorem, this is equivalent to many other conditions.  
In this paper, in the topologies we consider, every open set is a countable disjoint union of  basic open sets (used to define the topology), and each basic open set is also closed.
In these settings, weak convergence in distribution is equivalent to having, for each basic open set $U$, $\lim_{n\ra\infty} \Prob(X_n\in U) =\Prob(Y\in U)$
(see \cite[Proof of Theorem 1.1]{Liu2020}).

For an abelian group $A$, we have $\wedge^2 A$ is the quotient of $A\tensor A$ by the subgroup generated by elements of the form $a\tensor a$ (and for an $R$-module $A$, we define $\wedge^2_R A$ similarly with the tensor product over $R$) and
$\Sym^2 A$ is the quotient of $A\tensor A$ by the subgroup generated by elements of the form $a\tensor b-b\tensor a$.  

For a group (resp. profinite group) $G$ and elements $g_1,\dots\in G$, we write $\langle g_1,\dots \rangle$ for the normal subgroup 
(resp. closed normal subgroup)
generated by $g_1,\dots$.

For a function $f(x)$, we write $f(x)=O(g(x))$ to mean that there exists a constant $C$ such that for all $x$ such that $f(x)$ is defined, we have $|f(x)|\leq C g(x)$.

For a group $G$ with an action of a group $\Gamma$, we write $G^\Gamma$ for the invariants, i.e. elements of $G$ that are fixed by every element of $\Gamma$.

In the distributions of interest from number theory, there will usually be a random number field, or random elliptic curve, or some such object behind the scenes.
In these situations, there are a countable number of objects of interest (such as imaginary quadratic  number fields), and we consider some enumeration of them such that there are a finite number up to some bound $B$, and then we take a uniform random object up to bound $B$, and consider the limit of these distributions as $B\ra\infty$.  We do not wish to suggest that the uniform distribution is the only distribution on a finite set.  Indeed, the entire point of Section~\ref{S:Universality} is based on the fact that there are many non-uniform distributions.  Even beyond the question of the distribution on the objects up to bound $B$, there is still a question of which enumeration one takes and this can have interesting and important effects (e.g. see \cite{Wood2010, Bartel2020}).  However, since we are using the examples from number theory mainly as motivation, in this paper we will usually be very brief or not mention at all how exactly we take the random number theoretic objects.

\subsection*{Acknowledgments} 
Much of this paper is based on the author's work with co-authors Nigel Boston, Hoi Nguyen, Yuan Liu, Will Sawin, Weitong Wang, David Zureick-Brown, and we thank them for all the insights gained from these collaborations.  The outlook in this paper also benefited from useful conversations with Jordan Ellenberg, Aaron Landesman, Sam Payne, Bjorn Poonen, Philip Matchett Wood.  The author would like to thank Yuan Liu, Hoi Nguyen, and Will Sawin for useful feedback on an earlier draft of this paper. 
This work was done with the support of a Packard Fellowship for Science and Engineering, National Science Foundation Waterman Award DMS-2140043, and  National Science Foundation grant DMS-2052036.

\section{The moment problem}

 In probability, one often detects the distribution of a random variable by its moments, i.e. the averages of certain functions of the random variable.  Most classically, the moments of a random variable $X\in\R$ are the averages $\E(X^k)$, indexed by natural numbers $k$, and the (mixed) moments of a random variable 
$(X_1,\dots,X_n)\in\R^n$ are the averages $\E(X_1^{k_1}\cdots X_n^{k_n})$, indexed by $n$-tuples of natural numbers $(k_1,\dots,k_n)$.

The moment problem asks whether moments determine a unique distribution, and results on the moment problem, such as the following, are foundational in probability theory. 
\begin{theorem}[Carleman's condition]
Let $X$ be a random real number such that $M_k=\E(X^k)$ is finite for all integers $k\geq 0$.  Then  if
\begin{equation}\label{E:Carl}
\sum_{k\geq 1} M_{2k}^{-\frac{1}{2k}}=\infty,
\end{equation} 
then there is a unique distribution for a random real number $Y$ such that $\E(Y^k)=M_k$ for all $k\geq 0$.  In particular, if $M_k=O(e^{k})$, then \eqref{E:Carl} holds.
\end{theorem}
This kind of uniqueness result is useful in a situation when we have a conjectural distribution, know its moments, and then can prove some random variable is distributed as conjectured by showing it has those moments.  In other situations, we have an unknown distribution, compute its moments, and then recognize those as moments of a well-known distribution, and can use a uniqueness result to show our distribution matches the well-known one.  

In many applications we have not a single random variable, but rather a sequence of random variables, and we seek their limiting distribution.  For this, we require a uniqueness theorem that is \emph{robust}, in the sense that we can prove that a sequence of random variables whose moments converge to certain values must converge in distribution to a certain limit.  

Now we will clarify some, slightly informal, language to talk about different aspects of the moment problem.  Suppose we are considering random variables taking values in some set, and a sequence of real-valued functions $f_0,f_1,f_2,\dots$ on that set whose averages give the moments of the random variables.  We say we have \emph{uniqueness} in the moment problem for moments $M_k\in \R$, if the following holds: for any two random variables $X,Y$ under consideration, if for all $k$ we have $\E(f_k(X))=\E(f_k(Y))=M_k$, then $X$ and $Y$ have the same distribution.   We say we have \emph{robust uniqueness} in the moment problem for moments $M_k\in \R$, if the following holds: for any sequence of random variables $Y,X_1,X_2,X_3,\dots$ under consideration, if for all $k$ we have $\lim_{n\ra\infty}\E(f_k(X_n))=\E(f_k(Y))=M_k$, then the $X_n$ 
weakly converge in distribution to $Y$.  We have \emph{existence} in the moment problem for moments $M_k\in \R$, if we know there exists a random variable $X$ with  $\E(f_k(X))=M_k$ for all $k$, and we have \emph{construction} in the moment problem for moments $M_k\in \R$ if we have existence and can moreover explicitly describe $X$ by giving useful formulas for its distribution on enough subsets to generate the underlying $\sigma$-algebra.

\subsection{Robust uniqueness for random abelian groups}
For example, Fouvry and Kl\"uners \cite{Fouvry2006} determined the distribution of the $4$-ranks $(2C_B)[2]$, where $C_B$ is the $2$-Sylow subgroup of the class group of a random imaginary quadratic field as in Section~\ref{S:intro}, as $B\ra\infty$.
Fouvry and Kl\"uners proved the following result (which covers all aspects of the moment problem for certain average values).
\begin{theorem}[{{\cite[Theorem 1]{Fouvry2006a}}}]\label{T:FK}
If $p$ is a prime and $X_1,X_2,\dots$ are random finite dimensional $\F_p$-vector spaces such that for every integer $k\geq 0$, we have 
$$\lim_{n\ra\infty} \E(\#\Sur(X_n,\F_p^k))=1,$$ then for each integer $r\geq 0$, we have
$$
\lim_{n\ra\infty} \Prob(X_n\isom \F_p^r)=p^{-r^2}\frac{\prod_{j=r+1}^\infty (1-p^{-j})}{\prod_{j=1}^r (1-p^{-j})}.
$$
\end{theorem}
The distribution and averages in Theorem~\ref{T:FK} are known to occur as $\Prob(X_{CL}/pX_{CL}\isom \F_p^r)$ and  
$\E(\#\Sur(X_{CL}/pX_{CL},\F_p^k))$, respectively,
for the random group $X_{CL}$ introduced in Section~\ref{S:intro}
(see \cite[Theorem 6.3, Corollary 6.5]{Cohen1984}), so there does exist a random variable with these averages and its distribution can be explicitly described.
In the paper \cite{Fouvry2006}, Fouvry and Kl\"uners determined the averages $\E(\#\Sur((2C_B)[2],\F_2^k))$, and then applied Theorem~\ref{T:FK} to determine the distribution of $4$-ranks of class groups of imaginary quadratic fields (and did the analogous work for class groups of real quadratic fields).   

Fouvry and Kl\"uners actually write $\prod_{0\leq i <k}(p^{\rk_p(X_n)-p^i})$, and we have interpreted that as the number of surjective homomorphisms  $\#\Sur(X_n,\F_p^k)$.
In \cite{Fouvry2006a}, Fouvry and Kl\"uners translate the knowledge of the averages of $\prod_{0\leq i <k}(p^{\rk_p(X)-p^i})$ for all $k$ to
the knowledge of the averages of $p^{\rk_p(X)k}=\#\Hom(X,\F_p^k)$ for all $k$ (which can be done  by a finite sum over the subgroups of $\F_p^k$).  
These latter averages are the classical moments of the random number $p^{\rk_p(X)}=|X|$.
When our random groups get more complicated (and in particular non-abelian), we will not be able to capture the entire data of our groups so simply in a number, or even a sequence of numbers,  but the functions $\#\Sur(-,G)$ or $\#\Hom(-,G)$ will continue to be important and convenient functions whose averages we will call the moments (or Sur-moments, Hom-moments) of a random group. 
(See \cite[Section 3.3]{Clancy2015} for a discussion about the fact that the Hom-moments for finite abelian $p$-groups are classical mixed moments of certain numerical invariants of the groups.) 
 The relationship between the Hom-moments and the Sur-moments is  analogous to the relationship of the moments $\E(X^k)$ and the factorial moments
$\E(X(X-1)\cdots(X-k+1))$ of a random real number---knowledge of either kind of moments for $k\leq m$ easily gives knowledge of the other kind for $k\leq m$, and the choice of which to use mainly depends which is more convenient for the problem at hand.

Fouvry and Kl\"uners's proof of the robust uniqueness part of Theorem~\ref{T:FK} actually works whenever 
$$\E(\#\Hom(X,\F_p^k))=\E(|X|^k)=O(p^{k^2/2})$$
(see \cite[Proposition 3]{Fouvry2006a}), echoing the refrain that moments that do not grow too quickly determine a distribution. 
(Note that in this generality we are not claiming existence of a distribution, but only uniqueness.)
Such moments are too large to use Carleman's condition to conclude the distribution of $|X|$ as a real number, and indeed there are different distributions of real numbers that give the same moments with this order of growth (e.g. various distributions that have the same moments as the log-normal distribution).  However, in our setting of course  $|X|$ is restrained to be a power of $p$.  

For a random cyclic cubic field $K$, with class group $\bCl_K$ with $3$-torsion $\bCl_K[3]$, 
Klys \cite{Klys2020} found the asymptotic  moments of $\bCl_K[3]/\bCl_K[3]^{\Gal(K/\Q)}$, and then applied the more general form of Theorem~\ref{T:FK} to determine the limiting distribution of $\bCl_K[3]/\bCl_K[3]^{\Gal(K/\Q)}$.

Ellenberg, Venkatesh, and Westerland prove the following.
\begin{theorem}[{\cite[Proposition 8.3]{Ellenberg2016}}]\label{T:EVW}
If for each $n\geq 0$, we have a random abelian $p$-groups $X_n$ such that for every abelian $p$-group $A$ we have,
$$\lim_{n\ra\infty} \E(\#\Sur(X_n,A))=1,$$
then the $X_n$ weakly converge in distribution to $X_{CL}$, i.e. for every abelian $p$-group $B$, we have
$$
\lim_{n\ra\infty} \Prob(X_n\isom B)= \Prob(X_{CL}\isom B)=\frac{\prod_{i\geq 1} (1-p^{-i})}{|\Aut(B)|}.
$$
\end{theorem}
Ellenberg, Venkatesh, and Westerland use Theorem~\ref{T:EVW}, along with a determination of certain limiting moments of class groups of imaginary quadratic extensions of $\F_q(t)$, to prove that in a limit where the discriminant goes to infinity and then $q$ goes to infinity, that the $\ell$-Sylow subgroups of these class groups are as predicted by the Cohen-Lenstra heuristics for any odd prime $\ell$, as long as $\ell\nmid q-1$ \cite[Theorem 1.2]{Ellenberg2016}.
The work of Ellenberg, Venkatesh, and Westerland also particularly pioneered the idea that it is useful to consider these averages of surjection counts to be moments.  
 
If we would like to consider more general finite abelian groups, and also distributions that have other moments, we have the following theorem by the author.  (The cited results are stated with stronger bounds on the $M_A$, but one can see that all that is used in the proof is the hypotheses below.)
\begin{theorem}[see {\cite[Thm 8.3, proof of Cor 9.2]{Wood2017}}]\label{T:Wmom}
Let $P$ be a finite set of primes, and let $\mathcal{A}$ be the set of finite abelian $P$-groups.  
Let $M_A\in \R$ for each $A\in \mathcal{A}$ such that $M_A=O(|\wedge^2 A|)$.
Let $Y,X_1,X_2,\dots$ be random groups in $\mathcal{A}$.
If for every $A\in \mathcal{A}$, we have
$$
\lim_{n\ra\infty} \E(\#\Sur(X_n,A))= \E(\#\Sur(Y,A))=M_A,
$$
then the $X_n$ weakly converge in distribution to $Y$, i.e. for every $B\in \mathcal{A}$,
$$
\lim_{n\ra\infty} \Prob(X_n\isom B)=\Prob(Y\isom B).
$$
\end{theorem}
When $A=\F_p^k$, we have $|\wedge^2 A|=p^{k(k-1)/2}$, so we see a similar upper bound to that of Fouvry and Kl\"uners.
Theorem~\ref{T:Wmom} was applied in \cite{Wood2017} to determine the limiting distribution of the Jacobians (a.k.a. sandpile groups) of Erd\H{o}s--R\'{e}nyi random graphs, and by 
M\'esz\'aros \cite{Meszaros2020} to determine the limiting distribution of the Jacobians of random regular graphs.
M\'esz\'aros's result then had the striking corollary that the adjacency
matrix of a random regular graph is invertible with high probability, answering a long-standing open question that is not a priori about random groups at all.

If we consider a random finite abelian group $X$, without any condition on primes dividing its order, we have a uniqueness result by W. Wang and the author
as a corollary of Theorem~\ref{T:Wmom}.
\begin{corollary}[{\cite[Theorem 6.13]{Wang2021}}]
Let $M_A\in \R$ for each finite abelian group $A$ such that $M_A=O(|\wedge^2 A|)$.
Let $X,Y$ be random finite abelian groups,
If for every finite abelian group $A$, we have
$$
 \E(\#\Sur(X,A))= \E(\#\Sur(Y,A))=M_A,
$$
then $X$ and $Y$ have the same distribution, i.e. for every finite abelian group $B$,
$$
\Prob(X\isom B)=\Prob(Y\isom B).
$$
\end{corollary}
\begin{proof}
For a finite abelian group $C$, let $C_p$ denote its Sylow $p$-subgroup.
We have
$$
\Prob (X\isom A)=\lim_{z\ra\infty} \Prob (\prod_{p\leq z} X_p\isom \prod_{p\leq z} A_p).
$$
Then we can apply Theorem~\ref{T:Wmom} with $P$ the set of primes at most $z$ to conclude the corollary.
\end{proof}

However, for general finite abelian groups, robustness no longer holds (as it is possible the limit in $n$ cannot be exchanged with the limit in $z$).  As in \cite[Example 6.14]{Wang2021},
we can consider a random finite abelian group $X$, e.g. such that
$$
\Prob(X\isom A)=\frac{\zeta(2)^{-1}\zeta(3)^{-1}\zeta({4})^{-1}\cdots}{|A||\Aut A|},
$$
where $\zeta$ is the Riemann zeta function and we can also write $\zeta(2)^{-1}\zeta(3)^{-1}\zeta({4})^{-1}\cdots$ as a product over primes $\prod_{p} \prod_{i\geq 2}(1-p^{-i})$.  
(There is a random group with this distribution--see e.g. \cite[Proposition 2.1]{Wood2018}, and it is the limiting distribution predicted by Gerth's extension \cite{GerthIII1987a} of the Cohen-Lenstra heuristics for $2\bCl_K$, where $K$ is a random real quadratic field.) Then consider the random groups $X \times \Z/p\Z$ for each prime $p$.
For any finite abelian group $A$, we have $\lim_{p\ra\infty} \E(\#\Sur(X \times \Z/p\Z,A))=\E(\#\Sur(X,A))$ since for $p$ large enough $p\nmid |A|$.  Yet the limiting distribution of the 
$X \times \Z/p\Z$ is the zero distribution, i.e. for each $A$ we have $\lim_{p\ra\infty} \Prob(X \times \Z/p\Z\isom A)=0$.
This is in stark contrast to the  situation for random real numbers \cite[Theorem 30.2]{Billingsley1986}, where whenever the moments determine a unique distribution, they do so robustly.

\subsection{When uniqueness fails}
Another important example of distributions arising in number theory are those predicted by Poonen and Rains \cite{Poonen2012} as the 
asymptotic
distributions of $p$-Selmer groups of random elliptic curves.  
We consider two different random $\F_p$ vector spaces, with distributions given as follows
\begin{align}
\Prob(X_{odd}\isom \F_p^k) = \begin{cases}
 p^{-(k^2-k)/2}\frac{\prod_{j=0}^\infty (1-p^{-2j-1})}{\prod_{j=1}^k (1-p^{-j})}   & \textrm{$k$ odd}\\
0 & \textrm{$k$ even}
\end{cases}\label{E:PR}\\
\Prob(X_{even}\isom \F_p^k) = \begin{cases}
 p^{-(k^2-k)/2}\frac{\prod_{j=0}^\infty (1-p^{-2j-1})}{\prod_{j=1}^k (1-p^{-j})}   & \textrm{$k$ even}\\
0 & \textrm{$k$ odd}\notag
\end{cases}.
\end{align}
Poonen and Rains \cite{Poonen2012} conjecture that these are the limiting distributions of $p$-Selmer group of elliptic curves over $\Q$ of odd and even parity, respectively, and note \cite[Proposition 2.22(c)]{Poonen2012} that these distributions have the same moments, even though they are quite different distributions, supported on entirely disjoint sets of groups.
Indeed, there moments are as follows, and  we see that these cases are just beyond the bounds of the uniqueness results mentioned above.
\begin{theorem}\label{T:PRmom}
For each $k\geq 0$, we have
\begin{align*}
\E( \#\Sur(X_{odd},\F_p^k))&=\E( \#\Sur(X_{even},\F_p^k))=& &p^{(k^2+k)/2}, \textrm{ and}\\
\E(\#\Hom(X_{odd},\F_p^k))&=\E(\#\Hom(X_{even},\F_p^k))=& &p^{(k^2+k)/2}\prod_{j=1}^{k}(1+p^{-j}).
\end{align*}
\end{theorem}
\begin{proof}[Proof Sketch]
The Hom-moments  are shown in   \cite[Proposition 2.22(c)]{Poonen2012}.  
The Sur-moments can be found, in principle, by applying M\"obius inversion  to the Hom-moments.
However, the following argument is perhaps more practical.
The distributions of $X_{odd}$ and $X_{even}$ occur as the limiting distribution of cokernels of uniform random $n \times n$ alternating matrices over $\F_p$ (where $n$ is odd or even, respectively).  It is a general feature that for various computations it can be helpful, even for a known distribution, to recognize it as the limit of natural distributions.
We can see the claimed limit by counting exactly  how many alternating matrices over $\F_p$ have corank $k$ for each $k$ as in \cite[Proposition 3.8]{Lewis2011}
(see also \cite[Theorem 1.10]{Bhargava2015b}).    Then, one can make a simple argument to compute the limiting moments of these random cokernels as in \cite[Theorem 11]{Clancy2015} (which does the analogous thing for symmetric matrices), and use the explicit formulas for the distribution of the random cokernels for each $k$ and $n$ along with the dominated convergence theorem, as
in \cite[Theorem 10]{Clancy2015}, to deduce that the limiting moments of the random cokernels agree with the moments of $X_{odd}$ and $X_{even}$.
\end{proof}

However, in a setting as we have described, we could also use the additional information that we are looking for a distribution supported only on groups of even rank (or odd rank), along with the moments, to determine a distribution.  

One important motivation for the conjectures of Poonen and Rains was the result of
Heath-Brown \cite{Heath-Brown1994} determining the limiting distribution of $2$-Selmer groups of a random quadratic twist of the congruent number curve.  Heath-Brown showed that the limiting distribution for the quotient of the $2$-Selmer group by the $\F_2^2$ coming from the $2$-torsion points on the curve is the $X_{odd}$ distribution for twists $D\equiv 5,7 \pmod{8}$ (when the Selmer rank is odd), and the $X_{even}$ distribution for twists $D\equiv 1,3 \pmod{8}$  (when the Selmer rank is even).
Heath-Brown determined these distributions by first determining the moments and then proving a robust uniqueness result for the moment problem. Heath-Brown pointed out that it was surprising that these different distributions had the same moment, and proved the following robust uniqueness result, taking into account the parity.
\begin{theorem}[{{\cite[Lemma 18, proof of Theorem 2]{Heath-Brown1994}}}]\label{T:HB}
Let $M_0,M_2,\dots$ be non-negative real numbers such that $M_k=O(2^{k(k+1)/2})$. 
Let $Y,X_1,X_2,\dots$ be random even dimensional $\F_2$-vector spaces.
Then if for every even $k\geq 0$, we have
$$
\lim_{n\ra\infty} \E(\#\Hom(X_n,\F_2^{k}))= \E(\#\Hom(Y,\F_2^{k}))=M_k,
$$
then the $X_n$ weakly converge in distribution to $Y$, i.e. for every even $r$ we have,
$$
\lim_{n\ra\infty} \Prob(X_n\isom \F_2^{r})=\Prob(Y\isom \F_2^{r}).
$$
The statement also holds if we replace ``even'' with ``odd.''
\end{theorem}
Feng, Landesman, and Rains \cite{Feng2020} face a similar issue (in a slightly different context, where the random groups have fixed finite support of a given parity, but they only know half the moments) and use knowledge of the parity along with moments to determine the distribution of $n$-Selmer groups of elliptic curves of fixed height over $\F_q(t)$ as $q\ra\infty$.

Given the two distributions of $X_{odd}$ and $X_{even}$ on $\F_p$-vector spaces given in \eqref{E:PR}, one natural question is what are all the distributions on $\F_p$-vector spaces with those same moments.  We will now show that these (plus their linear combinations) are the only such distributions.

\begin{theorem}\label{T:only2}
Given non-negative reals $M_{-1},M_0,M_1,\dots$, and $p>1$, and $b<3$, such that $M_k=O(p^{\frac{k^2+bk}{2}})$, there is at most one simultaneous solution $(x_s)_s$ to 
\begin{align*}
\sum_{s=0}^{\infty} (-1)^s x_s &=M_{-1} \quad\quad \textrm{and}\\
\sum_{s=0}^{\infty} x_sp^{sk}&=M_k  \quad\quad k=0,1,\dots
\end{align*}
such that $x_s\geq 0$ for all $s$.
\end{theorem}
We note that this proof strategy is in the style of the earliest work on this problem, and not the more recent work, but it will also let us see some of the main features of the moment problem.
\begin{proof}
We modify the method from \cite[Lemma 18]{Heath-Brown1994}.
First, assuming we have a non-negative solution, we can bound $x_s$ using the $k=s$ equation to obtain
$$
x_s=O(p^{\frac{-s^2+bs}{2}}). 
$$
From this it follows that for any $N\geq 0$ and $k\leq N-2$,
$$
\sum_{s\geq N} x_sp^{sk} =O(\sum_{s\geq N} p^{\frac{-s^2+bs+2ks}{2}})=O(p^{\frac{-N^2+bN+2kN}{2}} ),
$$
where we allow the constant in the $O$ to depend on $p$.

We take some positive integer $N$, and we truncate the system to write
$$
\sum_{s=0}^{N-1} x_sp^{sk}=M'_k 
$$
for $k=-1,0,1,\dots N-2$ (except for $k=-1$ we replace $p^{sk}$ with $(-1)^s$).  Let $V$ be the $N\times N$ matrix whose $i,j$ coefficient is $p^{(i-2)(j-1)}$ for $i\geq 2$ and $(-1)^{j-1}$ for $i=1$.  Let $x$ be the vector with entries $x_0,\dots,x_{N-1}$ and $M'$ the vector with entries $M'_{-1},\dots, M'_{N-2}$.  Then $Vx=M'$.  (All of these implicitly depend on $N$.)  We will just give the first row of $V^{-1}$ explicitly.  Since $V$ is Vandermonde, we have $\det V=\prod_{0\leq i <j \leq N-2} (p^{j}-p^i) \prod_{i=0}^{N-2} (p^i+1)$.  Note that the $i,1$ minor of $V$ is also 
Vandermonde (after dividing out a factor from each row)
on the same elements, except for $p^{i-2}$, (or $-1$ when $i=1$). So we have
\begin{align}\label{E:invcoeff}
(V^{-1})_{1,j}=\frac{\pm p^{\frac{(N-2)(N-1)}{2}-(j-2)}}{(p^{j-2}+1)\prod_{\substack{0\leq i \leq N-2 \\ i\ne j-2}} (p^{j-2}-p^i)  }
\end{align}
for $j>1$, and 
$$
(V^{-1})_{1,1}=\frac{\pm p^{\frac{(N-2)(N-1)}{2}}}{ \prod_{i=0}^{N-2} (p^i+1)}
$$
and in all cases
$$
(V^{-1})_{1,j}=O(p^{\frac{-j^2+j}{2}}).
$$
So
\begin{align*}
x_0=&\sum_{j=1}^{N} (V^{-1})_{1,j} M'_{j-2}\\
=&\sum_{j=1}^{N} (V^{-1})_{1,j} M_{j-2}  +O(\sum_{j=1}^{N} p^{\frac{-j^2+j}{2}} |M_{j-2}-M'_{j-2}|  )  \\
=&\sum_{j=1}^{N} (V^{-1})_{1,j} M_{j-2}  +O( p^{\frac{(b-3)N}{2}}  ) .
\end{align*}
So that means $x_0$ must be $\lim_{N\ra \infty} \sum_{j=1}^{N} (V^{-1})_{1,j} M_{j-2} $ (where the matrix $V$ implicitly depends on $N$).  

Once $x_0$ is determined, we notice that our equations imply
\begin{align*}
\sum_{s=1}^{\infty} (-1)^{s-1} x_s =-(M_{-1}-x_0) \textrm{ and}\\
\sum_{s=1}^{\infty} x_sp^{(s-1)k}=(M_k-x_0)p^{-k}, 
\end{align*}
and we have a new system whose constants are still $O(p^{\frac{k^2+bk}{2}})$, and thus we can apply to same reasoning to deduce $x_1,\dots,$ each have at most 1 possible value.
\end{proof}

\begin{corollary}
If $\mu_{odd},\mu_{even}$ are the distributions of $X_{odd},X_{even}$, then any random $\F_p$-vector space $X$ such that for all $k$,
$$
\E(\#\Hom(X,\F_p^k))=p^{(k^2+k)/2}\prod_{j=1}^{k}(1+p^{-j})
$$
has distribution $\lambda \mu_{odd}+(1-\lambda) \mu_{even}$ for some $0\leq \lambda \leq 1$.
\end{corollary}
\begin{proof}
Clearly $\lambda \mu_{odd}+(1-\lambda) \mu_{even}$ give distributions with these same moments, and they each assign a different probability to the group being odd rank.  
Let $\lambda$ be the probability that $X$ has odd rank.  
We apply Theorem~\ref{T:only2} with $x_s=\Prob(X\isom \F_p^s)$, and $M_{-1}=1-2\lambda$, and $M_k=p^{(k^2+k)/2}\prod_{j=1}^{k}(1+p^{-j})$, and find that there
are unique values $x_s$ satisfying the equations, which proves the corollary.
\end{proof}

\begin{oproblem}
Besides the parity of the rank, are there other natural moments that we can consider for random finite $\F_p$-vector spaces, or finite abelian groups more generally, so that with the additional moments we can strengthen uniqueness results to allow for larger growing moments?
\end{oproblem}

In forthcoming work of Nguyen and the author, we prove  a generalization of the robust uniqueness result of Theorem~\ref{T:HB} for random finite abelian groups whose orders are supported on a finite set of primes, with a parity condition on the group.  
\begin{theorem}[Nguyen-W., forthcoming]\label{T:Hoimom}
Let $P$ be a finite set of primes, and let $\mathcal{A}$ be the set of finite abelian $P$-groups.  
Let $M_A\in \R$ for each $A\in \mathcal{A}$ such that $M_A=O(|\Sym^2 A|)$.
Let $a$ be an integer and $Y,X_1,X_2,\dots$ be random groups in $\mathcal{A}$, either 
\begin{enumerate}
\item all supported on groups of the form $G\times G$, or
\item all supported on groups of the form $\Z/a\Z \times G \times G$, for $G$ with $aG=0$.
\end{enumerate}
If for every $A\in \mathcal{A}$, we have
$$
\lim_{n\ra\infty} \E(\#\Sur(X_n,A))= \E(\#\Sur(Y,G))=M_A,
$$
then the $X_n$ weakly converge in distribution to $Y$, i.e. for every $B\in \mathcal{A}$,
$$
\lim_{n\ra\infty} \Prob(X_n\isom B)=\Prob(Y\isom B).
$$
\end{theorem}

\subsection{Random finite abelian groups with additional structure}
The class groups of Galois fields are not just abelian groups, but are also $\Z[G]$-modules, where $G$ is the Galois group.  
Let $\Z[G]'=\Z[G,|G|^{-1}]$.
Given a number field $k$ and a finite group $G$, the Cohen-Lenstra-Martinet heuristics \cite{Cohen1984,Cohen1990} give a distribution on $\Z[G]'$-modules, and conjecture that 
a random $G$-extension of $k$ has class group who prime-to-$|G|$ part is according to their distribution.  Thus for potential number theoretic applications, one would like robust uniqueness for the moment problem for random finite $\Z[G]'$-modules.
W. Wang and the author have given such a robust uniqueness result (the stated results are only for particular moments that occur in the Cohen-Lenstra-Martinet heuristics, but the proof works without change for the result given here).
\begin{theorem}[see {\cite[Theorem 6.11]{Wang2021}}]\label{T:Gmoduniq}
Let $G$ be a finite group.
Let $P$ be a finite set of primes, none dividing $|G|$, and let $\mathcal{A}$ be the set of finite $P$-group $\Z[G]'$-modules.  
Let $M_A\in \R$ for each $A\in \mathcal{A}$ such that $M_A=O(|\wedge_{\Z[G]'}^2 A|)$.
Let $Y,X_1,X_2,\dots$ be random $\Z[G]'$-modules in $\mathcal{A}$.
If for every $A\in \mathcal{A}$, we have
$$
\lim_{n\ra\infty} \E(\#\Sur_G(X_n,A))= \E(\#\Sur_G(Y,A))=M_A,
$$
then the $X_n$ weakly converge in distribution to $Y$, i.e. for every $B\in \mathcal{A}$,
$$
\lim_{n\ra\infty} \Prob(X_n\isom B)=\Prob(Y\isom B).
$$
\end{theorem}

Theorem~\ref{T:Gmoduniq} can be applied to work of Liu, Zureick-Brown, and the author \cite{Liu2019}, to prove, for every finite group $G$, a function field analog of the Cohen-Lenstra-Martinet heuristics for $G$-extensions over $\F_q(t)$, as $q\ra\infty$, as we will see below.  
Wang and the author \cite[Theorem 6.2]{Wang2021} have found the moments of the Cohen-Lenstra-Martinet distributions on $\Z[G]'$-modules.
In \cite{Liu2019}, we count and compare components of various Hurwitz schemes to estimate the moments of the class groups of random $G$-extensions of $\F_q(t)$, and notice those moments, in the limit where $q\ra\infty$ and then the degree $n$ of the (reduced) branch locus of the cover (i.e. the size of the radical of the discriminant) goes to infinity, match those predicted by Cohen-Lenstra-Martinet.  Theorem~\ref{T:Gmoduniq} then tells us that the limiting distribution of these class groups, when $q$ and $n$ both go to $\infty$, and $q$ is sufficiently large in terms of $n$, is as predicted by the Cohen-Lenstra-Martinet heuristics.  (Some caveats: these results are only in the case of extensions split completely over infinity, are only about the part of the class group that is prime to $|G|$, and $q$ must be taken so that $q-1$ is relatively prime to all the primes in $P$, and $q$ is prime to $|G|$ and the primes in $P$.  So these results do not see the part of the class group that is affected by roots of unity in $\F_q(t)$ \cite{Malle2008,Garton2015}.)  Precisely, we have the following.

\begin{theorem}[Corollary of {\cite[Corollary 1.5]{Liu2019}} and {\cite[Theorems 6.2 and 6.11]{Wang2021}}]\label{T:FFCM}
Let $G$ be a finite group and $P$ be a finite set of primes that are relatively prime to $|G|$.
Let $B$ be a finite  abelian $P$-group $\Z[G]$-module, and $B^G=0$.

Let $K_{q,n}$ be a uniform random Galois $G$-extension $K$ of $\F_q(t)$, split completely over $\infty$, with the norm of the radical of its discriminant $K/\F_q(t)$ at most $q^n$.  
Let $X_{q,n}$ be the product of the Sylow $p$-subgroups of the class group of $K_{q,n}$ (more precisely, of its ring of integers over $\F_q[t]$) for $p\in P$.

Then if $q_n$ is a sequence of prime powers growing sufficiently fast in $n$, such that for all $n$ we have 
that $q_n$ is relatively prime to $|G|$ and all the primes in $P$ and $q_n-1$ is relatively prime to all the primes in $P$,
then
$$
\lim_{n\ra\infty} 
\Prob(X_{q_n,n}\isom B )
 =\frac{c}{|B||\Aut_G(B)|},$$
where $c$ is a constant depending on $G$ and $P$ such that the limiting probabilities above sum, over $B$, to $1$.
\end{theorem}

\begin{proof}
By \cite[Corollary 1.5]{Liu2019},  for every finite abelian $P$-group $\Z[G]$-module $H$ with $H^G=0$, and every $\epsilon>0$, there is an $N_\epsilon$, such that for $n\geq N_\epsilon$, we have
$$
\left| \lim_{\substack{ q\ra\infty\\ (q,|G|)=1\\ (q(q-1),p)=1 \textrm{ for $p\in P$}}} \E(\#\Sur_G( X_{q,n},H  )) - |H|^{-1} \right| \leq \epsilon/2.
$$
For $n\geq N_\epsilon$, we choose a $Q_{n,\epsilon}$ such that  for  $q\geq Q_{n,\epsilon}$ (satisfying the conditions above) we have 
$$
 \left| \E(\#\Sur_G( X_{q,n},H  )) - |H|^{-1} \right| \leq \epsilon.
$$
So, if for each $n$, we consider the smallest $\epsilon$ such that $n\geq N_\epsilon$, and then take $q_n\geq Q_{n,\epsilon}$, we have
$$
\lim_{n\ra\infty} \E(\#\Sur_G( X_{q_n,n},H  ))=|H|^{-1}.
$$

Since $\bCl\O_K$ is trivial and $(|X_{q,n}|,|G|)=1$, we have $X_{q,n}^G=0$ \cite[Cor. 7.7]{Cohen1990}, so if $H$ is such that $H^G\ne 0$, we have
$\#\Sur_G( X_{q,n},H  )=0$.
By \cite[Theorem 6.2]{Wang2021}, we have that these are also the moments of the random $\Z[G]$-module $Y$ such that for any 
finite  abelian $P$-group $\Z[G]$-module $B$ with $B^G=0$ (on which $Y$ is supported)
$$
\Prob(Y\isom B)=\frac{c}{|B||\Aut_G(B)|},
$$
where $c$ is a constant depending only on $P$ and $G$.  Thus by Theorem~\ref{T:Gmoduniq} we conclude the theorem.
\end{proof}

As described by Wang and the author \cite[Sections 7-8]{Wang2021}, the class groups of non-Galois fields, away from certain bad primes, are also modules for a certain maximal order $\mathfrak{o}$ in a semi-simple algebra depending on the Galois group $G$ of the 
Galois closure over $\Q$ and over the field itself, and moreover are determined (as modules) from the class group of the Galois closure.
 The algebra $\mathfrak{o}$ can be non-trivial even when the non-Galois field has no automorphism.  
We can thus show that the Cohen-Lenstra-Martinet heuristics imply conjectures for the distribution of class groups of non-Galois fields.
For the part of the class group prime to $|G|$, analogous results to Theorem~\ref{T:FFCM} for the non-Galois case then follow formally from Theorem~\ref{T:FFCM} and the results in \cite{Wang2021}.
However, for non-Galois extensions, the ``bad'' primes avoided by the conjectures are not always all primes dividing $|G|$. 
So at certain ``good'' primes $p$ dividing $|G|$, we have shown in \cite[Theorem 8.14]{Wang2021} that the Cohen-Lenstra-Martinet heuristics imply a conjectural distribution on
the Sylow $p$-subgroups of class groups of non-Galois extensions (with Galois closure of group $G$) as well.  See \cite[Theorem 8.14]{Wang2021} for the relevant notion of good primes.  Here we mention a few examples of good primes: $2$ for $S_3$ cubic extensions, $3$ for $A_4$ and $S_4$ quartic extensions, $2$ for quintic $D_5$ or $A_5$ extensions.  The moment calculations and the unique robustness of the moment problem results in \cite{Wang2021} include the situations for all good primes for non-Galois extensions, as they are more generally for distributions of modules over maximal orders in semi-simple algebras.

In particular, the robust uniqueness result in \cite[Theorem 6.11]{Wang2021} is a version of Theorem~\ref{T:Wmom} in which $\Z[G]'$ is replaced by a maximal order in a semi-simple algebra.
Sawin \cite[Theorem 1.3]{Sawin2020} has proven an version of Theorem~\ref{T:Wmom}, in which $\Z[G]'$ is replaced by any associative algebra $R$ such that there are only finitely many isomorphism classes of finite simple $R$-modules, and $\operatorname{Ext}^1_R$ between any two finite $R$-modules is finite, but one requires the stronger assumption that $M_A=O(|A|^{O(1)})$.

As another example of additional structure, for the Sylow $p$-subgroups of class groups of quadratic extensions of $\F_q(t)$, Lipnowski, Sawin, and Tsimerman find that these groups have additional structure when $p^n\mid q-1$ \cite{Lipnowski2020} (where $q-1$ crucially is the number of roots of unity in $\F_q(t)$).  This structure involves two pairings and a compatibility relation, and they call a  group with such structure a $p^n$-Bilinearly Enhanced Group.  In \cite[Section 8]{Lipnowski2020}, they define moments for these enhanced groups and address the uniqueness and robustness aspects of the moments problem in this context.  They then apply their moment problem result, along with the homological stability results of Ellenberg, Venkatesh, and Westerland \cite{Ellenberg2016}, 
 to give a limiting distribution of Sylow $p$-subgroups of class groups of quadratic extensions of $\F_q(t)$, along with this extra structure.

\subsection{Random non-abelian groups}\label{SS:nonabmom}
One can also consider random non-abelian groups.  A natural such group arising in number theory is $\Gal(K^{un}/K)$, the Galois group of the maximal unramified extension of some random number field $K$.
We have that $\Gal(K^{un}/K)=\pi_1^{\acute{e}t}(\Spec \O_K)$ and this group has abelianization $\bCl_K$.  The maximal pro-$p$ quotient $G_p(K)$ of $\Gal(K^{un}/K)$ is the $p$-class tower group of $K$, the Galois group of $K^p$, the $p$-class tower of $K$.

Boston, Bush, and Hajir \cite{Boston2017a,Boston2021}, inspired by the Cohen-Lenstra heuristics, developed heuristics predicting the distribution of $G_p(K)$ for $K$ a random imaginary (respectively, real) quadratic field and $p$ an odd prime.  Boston and the author \cite{Boston2017} found the moments of the conjectural distribution of Boston-Bush-Hajir for imaginary quadratic fields, and prove robust uniqueness for the moment problem for these moments.  

Now, as we are considering random profinite groups, the set of isomorphism classes of groups under consideration is uncountable, and we need to be more precise about the measure theory.
For a quadratic field $K$, note that $G_p(K)$ has an action of $\Z/2\Z=\Gal(K/\Q)$, by lifting elements to $\Gal(K^{p}/\Q)$ and conjugating.  In general this would only be an outer action, but since $p$ is odd, by 
the Schur-Zassenhaus theorem we can find a splitting of $\Gal(K^{p}/\Q)\ra \Gal(K/\Q)$, and the resulting action of $\Gal(K/\Q)$ on $G_p(K)$ doesn't depend, up to isomorphism, on the choice of splitting.  
Let $\mathcal{G}_p$ be the set of isomorphism classes of finitely generated pro-$p$ groups with a continuous action of $\Z/2\Z$ (i.e. where morphisms must be equivariant for the $\Z/2\Z$ action).
A pro-$p$ group has a canonical lower $p$-central series defined by $P_0(G):=G$,  and for $n\geq 0$, we define $P_{n+1}(G)$ to be the closed subgroup generated by the commutators $[G,P_n(G)]$ and $P_n(G)^p$.  A finitely generated pro-$p$ group $G$ then has canonical finite quotients $Q_n(G):=G/P_n(G)$.  We let $\Omega$ be the $\sigma$-algebra on $\mathcal{G}_p$ generated by the sets 
$$
\{ G | Q_c(G)\isom P \},
$$
as $P$ ranges over $p$-groups.  We consider all random variables valued in $\mathcal{G}_p$ to be for the $\sigma$-algebra $\Omega$. (See \cite[Section 3]{Boston2017} for more details.)
With these preliminaries, we can state the uniqueness result of Boston and the author.
\begin{theorem}[{\cite[Theorems 1.3 and 1.4]{Boston2017}}]\label{T:momBBH}
Let $p$ be an odd prime.
There is a random $X_{BBH}\in \mathcal{G}_p$ whose distribution is the predicted distribution of Boston-Bush-Hajir for $G_p(K)$ for imaginary quadratic $K$.
For all finite $P\in \mathcal{G}_p$, we have
$$
\E(\#\Sur_{\Z/2\Z}(X_{BBH},P))=1,
$$
If we have a random $X\in \mathcal{G}_p$ such that for all finite $P\in \mathcal{G}_p$, we have
$$
\E(\#\Sur_{\Z/2\Z}(X,P))=1,
$$
then $X$ has the same distribution as $X_{BBH}$.
\end{theorem}

The argument in \cite{Boston2017} actually shows the following more general uniqueness result.
\begin{theorem}[{see \cite[Lemma 4.7, proof of Theorem 4.9]{Boston2017}}]\label{T:Mcond}
Let $p$ be a prime and $M_P\in \R$ for each finite $P\in \mathcal{G}_p$.
Let $\mathcal{G}^c_p$  the image of $\mathcal{G}_p$ under $Q_c$.  
Suppose that for each $c\geq 0$ and each $P\in\mathcal{G}^c_p$, we have
\begin{equation}\label{E:Mcond}
\sum_{Q\in \mathcal{G}^c_p} \frac{M_Q|\Sur_{\Z/2\Z}(Q,P)|}{M_P|\Aut_{\Z/2\Z}(Q)|}<2.
\end{equation}
If we have  random $X,Y\in \mathcal{G}_p$ such that for all finite $P\in \mathcal{G}_p$, we have
$$
\E(\#\Sur_{\Z/2\Z}(X,P))=\E(\#\Sur_{\Z/2\Z}(Y,P))=M_P,
$$
then $X$ and $Y$ have the same distribution.
\end{theorem}
The challenge in applying Theorem~\ref{T:Mcond} is that  it is not at all clear how one can evaluate the sum in \eqref{E:Mcond}.
Note that \eqref{E:Mcond} is  a sum of quite a different flavor than if we were considering abelian groups.  In particular, we don't have any convenient enumeration of all finite $p$-groups, and so evaluating this sum seems to involve a rather difficult group theory problem.
In \cite{Boston2017}, we prove that \eqref{E:Mcond} holds when $p$ is odd and all $M_P$ are $1$, but by a round-about argument that uses the construction of $X_{BBH}$.

In \cite{Boston2017}, we analyze components of certain Hurwitz schemes to prove that in a certain function field analog some of the moments of $G_p(K)$ for quadratic $K/\F_q(t)$ (ramified at infinity)
agree with the conjectures of Boston, Bush, and Hajir.  In our result \cite[Theorem 1.5]{Boston2017}, we let the degree of the discriminant go to infinity, and then let $q$ go to infinity, and as in Theorem~\ref{T:FFCM} we require that 
$(q,2p)=1$ and $(q-1,p)=1$.  This result involves the generally more difficult limit of letting $q$ go to infinity after the bound on the discriminant, as in the theorem of \cite{Ellenberg2016}, and we also use the theorem of Ellenberg, Venkatesh, and Westerland \cite{Ellenberg2016}  on the homological stability of Hurwitz spaces in the proof.

While Theorem~\ref{T:momBBH} certainly helps contextualize the result of \cite{Boston2017} on function field moments, it doesn't immediately apply, because Theorem~\ref{T:momBBH} proves only uniqueness and not robust uniqueness, which would be required in our desired applications, as they involve limits of distributions. 
 In the non-abelian setting,  Sawin recently proved a robust uniqueness result however that can be applied.

We will now explain what is required for this robust uniqueness result for non-abelian profinite groups.  Fix a finite group $\Gamma$, and consider the set $\mathcal{G}$ of isomorphism classes of profinite groups
with a continuous action of $\Gamma$, finitely many surjections to any finite group, and all continuous finite quotients of order relatively prime to $|\Gamma|$.  We will we define a topology on $\mathcal{G}$, introduced by Liu, Zureick-Brown, and the author \cite{Liu2019} (based on \cite{Liu2020}), and our $\sigma$-algebra $\Omega$ will be the Borel $\sigma$-algebra for that topology.  As we used $Q_c(G)$ above, we would like our topology to filter our profinite groups by certain canonical finite quotients.  
We will make such a canonical finite quotient for any finite set $\mathcal{C}$ of finite groups with an action of $\Gamma$ (we call these \emph{$\Gamma$-groups}).  Let $\bar{\mathcal{C}}$ be the closure of $\mathcal{C}$ under taking
$\Gamma$-equivariant subgroups, products, and quotients.  Let $G^{\mathcal{C}}$ be the inverse limit of all quotients of $G$ that are in $\bar{\mathcal{C}}$.
Then these $G^{\mathcal{C}}$ (indexed by finite sets $\mathcal{C}$ of finite groups) are the canonical quotients we will use.  We then use the topology on $\mathcal{G}$ whose open sets are generated by
$$
\{ G | G^{\mathcal{C}}\isom H \},
$$
where $H$ ranges over all finite $\Gamma$-groups. 
Then Sawin's robust uniqueness result can be stated as follows.

\begin{theorem}[{\cite[Theorem 1.2]{Sawin2020}}]\label{T:Sawin}
Let $\Gamma$ be a finite group and $\mathcal{C}$ be a finite set of finite $\Gamma$-groups whose orders are relatively prime to $|\Gamma|$.  
For every finite $\Gamma$-group $H$, let  $M_H\in \R$ such that $M_H=O(|H|^{O(1)})$.
Let $Y,X_1,X_2,\dots$ be random groups in $\mathcal{G}$.  Assume that for every finite  $\Gamma$-group $H$ with $H^\mathcal{C}=H$, we have
$$
\lim_{n\ra\infty} \E(\#\Sur_\Gamma(X_n,H))=\E(\#\Sur_\Gamma(Y,H)),
$$
Then for every finite  group $H$ with an action of $\Gamma$,
\begin{equation}\label{E:Clim}
\lim_{n\ra\infty} \Prob( X_n^{\mathcal{C}}\isom H )=\Prob( Y^{\mathcal{C}}\isom H ).
\end{equation}
\end{theorem}

\begin{corollary}\label{C:Sawin}
Let $\Gamma$ be a finite group.  
For every finite $\Gamma$-group $H$, let  $M_H\in \R$ such that $M_H=O(|H|^{O(1)})$.
Let $Y,X_1,X_2,\dots$ be random groups in $\mathcal{G}$.  Assume that for every finite  $\Gamma$-group $H$, we have
$$
\lim_{n\ra\infty} \E(\#\Sur_\Gamma(X_n,H))=\E(\#\Sur_\Gamma(Y,H)).
$$
Then the distributions of the $X_i$ weakly converge to the distribution of $Y$.
\end{corollary}

Sawin proved Theorem~\ref{T:Sawin} in order to apply it to results of Liu, Zureick-Brown and the author \cite{Liu2019}.  We discussed above that the moments of the class groups of random $\Gamma$-extensions $K/\F_q(t)$ were found in the paper \cite{Liu2019} (as $q\ra \infty$), but  this paper found, more generally, the moments of 
$\Gal(K^\#/K)$, where $K^\#$ is the maximal unramified extension of $K$ that is prime to $|\Gamma|$, prime to $q(q-1)$, and split completely at infinity
\cite[Theorem 1.4]{Liu2019}.
Moreover, the paper constructed a distribution on random groups with these moments \cite[Theorem 1.2,Theorem 6.2]{Liu2019}.  Sawin's result then can be applied \cite[Theorem 1.1]{Sawin2020} to conclude that (in a limit where $q\ra\infty $ fast enough compared to $n$, similar to Theorem~\ref{T:FFCM}) 
the random profinite groups $\Gal(K^\#/K)$ converge in distribution to the  group constructed in \cite{Liu2019}. 

For quadratic extensions  $K/\F_q(t)$, we can apply the work of Liu, Zureick-Brown and the author \cite{Liu2019}, the homological stability result of Ellenberg, Venkatesh, and Westerland \cite{Ellenberg2016}, and Sawin's result Theorem~\ref{T:Sawin}, and find the limiting distribution of the maximal unramified odd extension of $K$ when $q,n\ra\infty$ in any way.
Let $X$ be a random  profinite group with an action of $\Z/2\Z$ with distribution $\mu_1$ from \cite[Section 4]{Liu2019} (with $\Gamma=\Z/2\Z$).
The measure of this distribution on basic opens is given explicitly in \cite[Equation (4.14)]{Liu2019}.
Let $\mathcal{F}_m$  be the free odd profinite group on $m$ generators, with a $\Z/2\Z=\langle \sigma \rangle$ action inverting each of the generators,
and let $y_i$ be independent random elements of $\mathcal{F}_m$ from Haar measure.
Then in \cite[Section 3]{Liu2019}, it is shown that $\mathcal{F}_m/\langle y_1^{-1}\sigma(y_1),\dots  y_{m+1}^{-1}\sigma(y_{m+1}) \rangle $
converge in distribution to $X$, as $m\ra\infty$.
Let $X_P$ be the pro-$P$ completion (i.e. the inverse limit of all the finite $P$-group quotients)
of $X$.

\begin{theorem}
Let $P$ be a finite set of odd primes.
Let $K_{q,n}$ be a uniform random  quadratic extension $K$ of $\F_q(t)$, split completely over $\infty$, with $\Nm \Disc K/\F_q(t)\leq q^n$.  
Let $K^P$ be the maximal unramified extension of $K$, split completely at infinity, all of whose finite subextensions have degree a product of primes in $P$.
Let $X_{q,n}= \Gal(K_{q,n}^P/K_{q,n}).$

Then as $q,n\ra\infty$ in any way such that $q$ is odd, relatively prime to the primes in $P$, and $q-1$ is relatively prime to the primes in $P$, 
then 
$$X_{q,n} \textrm{ converge in distribution to }X_P.$$
\end{theorem}

\begin{proof}
Let $\Gamma=\Z/2\Z$.
We follow \cite[Proof of Theorem 1.4]{Liu2019}, but will use the homological stability result of Ellenberg, Venkatesh, and Westerland \cite{Ellenberg2016}.
Let $H$ be a finite $P$-group with an action of $\Gamma$, such that the co-invariants $H_\Gamma$ are trivial (note this is equivalent to the admissibility condition in \cite{Liu2019}, given the condition on $P$).  

Let $q$ be a prime power relatively prime to $2$ and all the primes in $P$, and let $q-1$ be relatively prime to all the primes in $P$.  
Let $E_\Gamma(n,q)$ be the set of quadratic extensions $K/\F_q(t)$, split completely at infinity, with $\Nm \Disc K/\F_q(t)=q^n$.
Note $n$ must be even for there to exist such a $K$ (e.g. by the Riemann-Hurwitz formula).
Let $G=H\rtimes \Gamma$.  Let $c$ be the set of elements of $G$ of order $2$, and note by the Schur-Zassenhaus Theorem this is a single conjugacy class of 
$G$.
Then there are Hurwtiz schemes $\cHur_{ G,c}^n$, $\cHur_{ \Gamma,\Gamma\setminus\{1\}}^n$ constructed in \cite{Liu2019}, such that by \cite[Lemma 10.2]{Liu2019}
$$
[H:H^\Gamma] \sum_{K\in E_\Gamma(n,q) } \#\Sur_\Gamma (\Gal(K^P/K),H)  = \#\cHur_{ G,c}^n(\F_q)
$$ 
and
$$
\#E_\Gamma(n,q) =\#\cHur_{ \Gamma,\Gamma\setminus\{1\}}^n(\F_q).
$$
For $n$ sufficiently large given $G$, by \cite[Theorem 10.4]{Liu2019}, we have that $\#\cHur_{ G,c}^n$ and
$\#\cHur_{ \Gamma,\Gamma\setminus\{1\}}^n$ have the same number, $z_n$, of Frobenius fixed components over $\bar{\F}_q$.
Moreover, $z_n$ is positive for even $n$ because we know $\F_q(t)$ has quadratic extensions split completely at infinity
and so $\#\cHur_{ \Gamma,\Gamma\setminus\{1\}}^n$ has $\F_q$-points.
By the Grothendieck-Lefschetz trace formula, we have
$$
|\#\cHur_{ G,c}^n(\F_q)-z_nq^n|\leq \sum_{j=0}^{2n-1}  q^{j/2} \dim H^j_{c,\acute{e}t}((\cHur_{ G,c}^n)_{\bar{\F}_q},\Q_\ell),
$$
for some $\ell$ (\cite[Lemma 10.3]{Liu2019} tells us $(\cHur_{ G,c}^n)_{\bar{\F}_q}$ is smooth and $n$-dimensional).
By \cite[Lemma 10.3]{Liu2019}, we then have 
$$
|\#\cHur_{ G,c}^n(\F_q)-z_nq^n|\leq \sum_{j=0}^{2n-1} q^{j/2} \dim H^{2n-j}((\cHur_{ G,c}^n)_{\bC},\Q).
$$ 
By \cite[Theorem 6.1, Proposition 2.5]{Ellenberg2016} (their $\operatorname{CHur}^{c}_{G,n}$ is the topological space of the analytic topology of our 
 $(\cHur_{ G,c}^n)_{\bC}$ by \cite[Section 11.3]{Liu2019}, and we can easily check their non-splitting condition is satisfied here), 
there exist constants $C$ and $D$, depending on $G$, such that   
$\dim H^{k}((\cHur_{ G,c}^n)_{\bC},\Q)\leq CD^{k}$.
Thus
 we have 
 $$
|\#\cHur_{ G,c}^n(\F_q)-z_nq^n|\leq \sum_{j=0}^{2n-1} q^{j/2} C D^{2n-j}.
$$
For $q\geq D^4$, we have 
$$
|\#\cHur_{ G,c}^n(\F_q)-z_nq^n|\leq \sum_{j=0}^{2n-1} C q^{n/2 + j/4} \leq \frac{2Cq^{n-1/4}}{1-q^{-1/4}}
$$ 
By the same argument, we have the same inequalities for $\cHur_{ \Gamma,\Gamma\setminus\{1\}}^n$
Summing over even $n\leq N$, we conclude that if $q,n\ra\infty$ in any way, we have 
$$
\frac{1}{\#E_\Gamma(n,q)}\sum_{K\in E_\Gamma(n,q) } \#\Sur_\Gamma (\Gal(K^P/K),H)  \ra [H:H^\Gamma]^{-1}.
$$
By \cite[Theorem 6.2]{Liu2019}, we see these are exactly the moments of the random $\Gamma$-group $X_P$ described above.  Thus applying Theorem~\ref{T:Sawin}, we conclude the result.
%
\end{proof}

The methods of the paper  \cite{Liu2019} can find the moments of the maximal unramified extension of a random $\Gamma$ extension $K/\F_q(t)$ even when we allow parts not prime to $q-1$, but the obstruction to proceeding is that there is no candidate conjectural random group with those moments.  This brings us to the first case in this story when there was not an already known conjectural distribution that one was trying to show some distributions from number theory converged to.  So we naturally turn to the existence and construction aspects of the moment problem.  

All of the questions on moment problems for random groups discussed above have been reducible to questions of a countable list of linear equations in 
a countable number of variables, and whether they have a unique solution. 
The equations and variables are parametrized by groups, and the coefficients are given by group theoretic quantities (numbers of surjective homomorphisms).
 In Theorem~\ref{T:only2}, we made these equations quite explicit, and inverted the implicit infinite matrices by truncating them to finite matrices that we could explicitly invert.  This is an approach that works well when the groups involved are $\F_p$-vector spaces, but it becomes less and less tractable as the 
 groups get more complicated.  For finite abelian groups, one relies on the classification of the groups and the ability to write a formula for the number of surjections from one to another.   For non-abelian groups, there is no reasonable formulaic parametrization of the groups and their numbers of surjections.
Theorem~\ref{T:Sawin} is proved by a localization process that reduces the question to one only involving a smaller list of groups that be classified and for which the number of surjections can be simply expressed.
 
 All of these proofs of uniqueness, at least in principle, give some expression for the (only possible) solutions to these systems of equations.
What then remains of the existence question?   (1) The solutions must be non-negative in order to describe a measure.  (2) The determined values must further be shown to satisfy the equations.
 (3) In some cases, the solutions must be compatible in order to describe a measure.

We elaborate a bit on what these remaining problems are like.
First we consider (1).  In Theorem~\ref{T:only2}, we find an expression for $x_0$, the probability of the trivial group, as
$$\lim_{N\ra\infty} \sum_{j=1}^{N} (V^{-1})_{1,j} M_{j-2},$$
where the $M_j$ are the given moments, and the coefficients of the inverse matrix are given explicitly in  \eqref{E:invcoeff}.
The other $x_i$ are given similarly, with modified values of $M_j$.  It is not clear whether one should expect a simple criterion for whether these values are non-negative, but it seems conceivable that 
for a particular nice family of $M_j$ of interest that one could, with work, prove the values of the $x_i$ that are determined are indeed positive.  Addressing (2), one could hope  to prove for sufficiently bounded moments that these determined values satisfied the equations.  We cannot see problem (3) above when the random groups are just $\F_p$-vector spaces, but even in the case of finite abelian $p$-groups, some approaches prove that the distribution on groups mod $p$ is determined, and then that the distribution on groups mod $p^2$ is determined, etc.  One can see this feature explicitly in the statement of Theorem~\ref{T:Sawin}.  So, in such cases, to prove existence, one would have to check that the determined values were compatible and could be pieced together into a probability distribution.

The \emph{construction} problem, which we have described above as giving \emph{useful} formulas for the distribution, now turns on what useful means.  The formulas for the distributions that arise from the uniqueness proofs above are generally infinite sums.
One might not expect to solve this for general moments, but perhaps only for specific moments that arise in particular problems.
  We propose as one test of usefulness---can one detect if the distribution assigns value 0 to any particular basic open set?
Note that the distributions on finite abelian groups we have seen above in Theorems~\ref{T:FK} and \ref{T:EVW} and in \eqref{E:PR} all have this property.  
The distributions on non-abelian groups we have discussed, including those of Boston, Bush, and Hajir, and Liu, Zureick-Brown, and the author also have this property (see \cite[Lemma 4.8]{Boston2017}, \cite[Theorem 4.12]{Liu2019}).
    Other tests for usefulness may come from the features of the desired application, but we emphasize that there can be a significant gap between having a formula for a distribution as an infinite sum, and being able to use that formula in practice to answer questions about the distribution.  

We mention briefly forthcoming work of Sawin and the author on the moment problem for profinite groups.  This work will strengthen Theorem~\ref{T:Sawin} so that that larger growing moments $M_H$ are allowed, up to the point where the statement is no longer true (e.g. because of the example ~\eqref{E:PR}).  We also prove a general existence result addressing the problems (2) and (3) mentioned above.
Our first applications are to problems where moments are known but the distribution is not known.  The first of these applications is mentioned above, and is for the distribution of class groups or their non-abelian analogs, or order not prime to roots of unity in the base field $\F_q(t)$.  The second is to the distribution of the profinite completion of random $3$-manifolds (from random Heegaard splittings), as introduced by Dunfield and Thurston \cite{Dunfield2006}.  In these applications, we also solve the construction problem, e.g. we can describe explicitly the support of the limiting distribution, and we can use our formulas for the limiting distribution to answer open questions about the distributions from number theory and topology.  Moreover, the 3-manifold application requires addressing situations where uniqueness does not actually hold, and we recover uniqueness with additional parity hypotheses, such as in Theorems~\ref{T:HB},  \ref{T:only2}, and  \ref{T:Hoimom} above.


\section{Universality}\label{S:Universality}
A central concept in probability theory is that of \emph{universality}, which describes the ubiquitous phenomena that many input independent distributions can be combined to make an output distribution, and  
as the number of input distributions goes to infinity, the output distribution because quite insensitive to the input distributions.
The first and most well-known example is the Central Limit Theorem.

\begin{theorem}[Central Limit Theorem]
Let $X_1,X_2,\dots$ be independent, identically distributed random real numbers with finite mean $\mu=\E(X_i)$ and finite variance $\sigma^2$.  Then as $n\ra \infty,$
$$
\sqrt{n} \left(\frac{X_1+\cdots+X_n}{n}-\mu \right) 
$$
converge in distribution to the normal distribution with mean $0$ and variance $\sigma^2$.  
\end{theorem}

Here the $X_i$ are the input distributions, and their normalized sum is the output distribution, and we see that the output, asymptotically, only depends on the variance of the input distributions.
The Central Limit Theorem is the tip of the iceberg, and probability theory is filled with further examples of this kind of phenomenon.

Here we discuss a somewhat newer line of inquiry:  universality for random groups.  In this case, the output distribution should be a random group, and the random group is somehow built out of the input distributions.    One natural way to obtain such a random group is to start with a fixed random group $F$ and take the quotient by random elements of $F$ that we call relators.  
If $F$ is a free abelian group, $F=\Z^n$, and we collect $m$ random relators as the columns of a matrix $M$, then the quotient of $F$ by our relators is the cokernel $\cok M$ (by definition of the cokernel).
This shows that questions about random abelian groups built in this way can be rephrased as questions about cokernels of random integral matrices.  

\subsection{Random finite abelian groups}
The simplest sort of groups to consider, as in our discussion above on the moment problem, are $\F_p$-vector spaces.  Let $F=\F_p^n$.  If $M$ is an $n\times m$ matrix with coefficients in $\F_p$, then the quotient of $F$ by the columns of $M$, i.e. $\cok M$, has rank equal to $n-\operatorname{rank} M=\operatorname{corank} M.$  Hence we translate questions about random $\F_p$-vector spaces into questions about ranks of random matrices over $\F_p$.    We note here that determining the rank distribution of random matrices over $\F_p$ is a simple exercise if the matrices are uniformly  distributed.  The entire interest here is when the matrix coefficients (still independent) are drawn from a wide range of distributions, and in particular if there is a resulting universality in the distribution of the ranks. 
 There is a long history of work on this question.
Kozlov \cite{Kozlov1966}  showed a universality result for the ranks over $\F_2$, and Kovalenko and Levitskaja \cite{Kovalenko1975} showed a version over $\F_p$.
Both works require that the matrix entries take all possible values with positive probability.  
Charlap, Rees, and Robbins \cite{Charlap1990} only determined the probability that a square matrix is invertible, but allowed more general matrix entries.
Balakin \cite{Balakin1968}, Bl\"omer, Karp, and Welzl \cite{Blomer1997}, and Cooper \cite{Cooper2000a} determined the ranks for sparser matrices,  with entries uniformly distributed over non-zero values.  
  The most general result we know is the following result of Nguyen and the author.

\begin{theorem}[Corollary of {\cite[Theorem 4.1]{Nguyen2021}}]\label{T:NWmodp}
Let $p$ be a prime.
Let $u$ be a non-negative integer and $\alpha_n$ a function of integers $n$ such that for any constant $\Delta>0$, for $n$ sufficiently large we have
$\alpha_n\geq \Delta (\log n)/n$.  For every positive integer $n$,
let $M_n$  be a random $n \times(n+u)$ matrix with independent entries $\xi_{i,j,n}\in \F_p$ that satisfy
$$
\max_{a\in \F_p} \Prob(\xi_{i,j,n}=a) \leq 1-\alpha_n
$$ for every $i,j,n$.  
Then for every $r\geq 0$,
$$
\lim_{n\ra\infty} \Prob(\cok M_n\isom \F_p^r)=\lim_{n\ra\infty} \Prob(\operatorname{rank} M_n=n-r)=p^{-r(r+u)}\frac{\prod_{j=r+u+1}^\infty (1-p^{-j})}{\prod_{j=1}^r (1-p^{-j})}.
$$
\end{theorem}
We see that there  are separate universality classes for different $u$, i.e. different numbers of relations compared to the number of generators, but for fixed $u$ a wide range of entry distributions all give random groups in the same universality class.
Note that Theorem~\ref{T:NWmodp} does not require the matrix entries to be identically distributed.  It also allows the matrices to be quite sparse.
If $\Prob(\xi_{i,j,n}=0)=1-(\log n)/n$ the matrix would have a row of all zeroes with (asymptotically) positive probability, and this crosses a threshold for the behavior of the random matrix, similar to the well-known threshold for the behavior of random graphs and sparse random matrices in other contexts.

\begin{oproblem}
Lower the bound on $\alpha_n$ in Theorem~\ref{T:NWmodp}, as close to the $(\log n)/n$ threshold as possible (and similarly for Theorems~\ref{T:NWpadic} and \ref{T:NWmain} below).
\end{oproblem}

We next consider finite abelian $p$-groups, and now $F=\Z_p^n$ (and $\Z_p$ are the $p$-adic integers).    If we form a random group by taking $n+u$ random relators, then the group is $\cok M$, where 
$M$ is the matrix whose columns are the relations.  Indeed, Theorem~\ref{T:NWmodp} is actually a corollary of the following.
\begin{theorem}[{\cite[Theorem 4.1]{Nguyen2021}}]\label{T:NWpadic}
Let $p$ be a prime.
Let $u$ be a non-negative integer and $\alpha_n$ a function of integers $n$ such that for any constant $\Delta>0$, for $n$ sufficiently large we have
$\alpha_n\geq \Delta (\log n)/n$.  For every positive integer $n$,
let $M_n$  be a $n \times(n+u)$ matrix with independent entries $\xi_{i,j,n}\in \Z_p$ that satisfy
$$
\max_{a\in \F_p} \Prob(\xi_{i,j,n}\equiv a \pmod{p}) \leq 1-\alpha_n
$$ for every $i,j,n$.  
Then for every abelian $p$-group $A$, we have
\begin{align*}
\lim_{n\ra\infty} \P(\cok(M_n)\isom A) 
&= \frac{1}{|A|^u|\Aut(A)|}
\prod_{k=1}^\infty (1-p^{-k-u}).
\end{align*}
\end{theorem}

The proof of Theorem~\ref{T:NWpadic} builds heavily on the method in \cite{Wood2019a}, but extends the statement to include the sparse regime.

\begin{proof}[Proof of Theorem~\ref{T:NWmodp}]
The probabilities in Theorem~\ref{T:NWpadic}  sum over $A$ to $1$ to give a probability distribution for each $u$ \cite[Lemma 3.2]{Wood2019a}.   Thus it follows from Fatou's Lemma that we can simply add up the probabilities from Theorem~\ref{T:NWpadic} for groups of rank $r$  to obtain the limiting probabilities in Theorem~\ref{T:NWmodp}.
This is done in \cite[Corollary 6.5]{Cohen1984}.
\end{proof}
  
When $u=0$, the distribution in Theorem~\ref{T:NWpadic} is the Cohen-Lenstra distribution of $X_{CL}$ we have mentioned above, and when $u=1$ it is the  distribution conjectured by Cohen and Lenstra 
\cite{{Cohen1984}}
for the Sylow $p$-subgroups  of class groups of real random quadratic fields (for $p$ odd).
  Let us now put these class groups in the context of random matrices, following Venkatesh and Ellenberg \cite[Section 4.1]{Venkatesh2010}.  
Let $K=\Q(\sqrt{D})$ for some negative (resp., positive) square-free integer $D$, and $S$ be any finite set of primes of $K$ that generate $\operatorname{Cl}(K)$.  
We write $\cO_S^*$ for the $S$-units in the integers $\cO_K$, and $I_K^{S}$ for the abelian group of fractional ideals generated by the elements of $S$.
Then 
\begin{equation}\label{E:casm}
\operatorname{Cl}(K)=\cok( \cO_S^* \ra I_K^{S} ),
\end{equation}
where the map takes $\alpha$ to the ideal $(\alpha)$.
So the Sylow $p$-subgroup of $\operatorname{Cl}(K)$ is $\cok( \cO_S^* \tensor_\Z \Z_p \ra I_K^{S} \tensor_\Z \Z_p)$.
Since  $I_K^{S}$ and $\cO_S^*$ are both abelian groups of rank $|S|$ (resp. of ranks $|S|$ and $|S|+1$), we have written 
the Sylow $p$-subgroup of $\operatorname{Cl}(K)$ as a cokernel of a $p$-adic $n\times n$ matrix $R_D$ (resp.  $n\times(n+1)$ matrix). 
One can now view the Cohen-Lenstra conjecture for class groups of quadratic fields as asking whether universality of Theorem~\ref{T:NWpadic} extends to the random matrix $R_D$ for random $D$.
This point of view was a motivation for the paper \cite{Wood2019a}.

Now we consider random finite abelian groups more generally.  For a finite set $P$ of primes, considering finite abelian $P$-groups turns out to be only notationally more challenging than considering abelian $p$-groups, and indeed  \cite[Theorem 4.1]{Nguyen2021} is proven in this slightly more general context.
However, considering all primes at once is quite a bit more of a challenge, because there will always be primes large compared to $n$.  Nguyen and the author develop a method to handle large primes (compared to $n$) and we prove the following.

\begin{theorem}[{\cite[Theorem 1.1]{Nguyen2021}}]\label{T:NWmain}
For integers $n,u\geq 0$, let $M_{n \times (n+u)}$ be an integral $n \times (n+u)$ matrix with entries 
 i.i.d. copies of a random integer $\xi_n$,  such that for every prime $p$ we have
$$
\max_{a\in \F_p} \Prob(\xi_{n}\equiv a \pmod{p}) \leq 1-n^{-1+\epsilon}
$$ 
 and  $|\xi_n|\leq n^T$ for any fixed parameters $0<\epsilon<1$ and $T>0$ not depending on $n$. 
For any fixed finite abelian group $A$ and $u\geq 0$,
\begin{equation}\label{eqn:sur:fix}
\lim_{n\to \infty}  \P\Big(\cok(M_{ n \times (n+u)}) \simeq A \Big)  = \frac{1}{|A|^u |\Aut(A)|}  \prod_{k=u+1}^\infty \zeta(k)^{-1},
\end{equation}
where $\zeta(s)$ is the Riemann zeta function.
\end{theorem}

Note this theorem has nice corollaries like the probability that a random map as in Theorem~\ref{T:NWmain} (for $u=1$) from $\Z^{n+1}\ra \Z^n$ is surjective is $\prod_{k=2}^\infty \zeta(k)^{-1}\approx 0.4358.$  As in the proof of Theorem~\ref{T:NWmodp}, one can obtain other probabilities as corollaries, such as (for $u\geq 1$) the probability that $\cok(M_{ n \times (n+u)})$ is cyclic.
However, when $u=0$, the probabilities in Theorem~\ref{T:NWmain} are all $0$ (from the $\zeta(1)^{-1}$ term), so this theorem tells us little about the distribution of random abelian groups from $n$ generators and $n$ random relations.  In \cite[Theorem 1.2]{Nguyen2021} we do find the probability that $\cok(M_{ n \times n})$ is cyclic, and in 
\cite[Theorem 2.4]{Nguyen2021} more generally give the probability that $\cok(M_{ n \times n})$ is any set of groups $\{A\times C | C \textrm{ cyclic, }p\nmid |C| \textrm{ for }1<p<Y \}$. 
However, we are not able to distinguish a factor of $\Z/p\Z$ for large $p$ from one of $\Z/p^2\Z$, for example.

\begin{oproblem}
Find $$\lim_{n\ra\infty} \Prob (|\cok M_{n\times n}| \textrm{ is square-free})=\lim_{n\ra\infty} \Prob (|\det M_{n\times n}| \textrm{ is square-free}).$$
\end{oproblem}
Note that finding the probability that a polynomial takes square-free values on even the nicest distributions of integers is difficult and generally open, but there has been some progress for certain
discriminant polynomials by Bhargava \cite{Bhargava2014b} and Bhargava, Shankar, and Wang \cite{Bhargava2016a}

\begin{oproblem}
Extend Theorem~\ref{T:NWmain} to non-identical entries.
\end{oproblem}

The first connection of the Cohen-Lenstra heuristics to random matrices came from work of Friedman and Washinton \cite{Friedman1989}.  They considered the  analog of the Cohen-Lenstra conjectures for quadratic extensions of $\F_q(t)$.  In this case one can also describe the Sylow $p$-subgroup of the class group of $K$ (or more precisely of the $\Pic^0$) as the cokernel of a certain random $2g\times 2g$ random matrix $I-F$ over $\Z_p$, where $I$ is the identity matrix, and $F$ describes  the action of Frobenius on the $p$-adic Tate module of the curve corresponding to $K$ \cite[Proposition 2]{Friedman1989}.  (Here $p$ is \emph{not} the characteristic of $\F_q$.)
Friedman and Washington showed that the cokernels of random matrices from the (additive) Haar measure on $n\times n$ matrices over $\Z_p$, as $n\ra\infty$, approach the Cohen-Lenstra distribution.
However, the matrix $F$ above is not just any matrix; since it acts on the Weil pairing by scaling the pairing by $q$, it lies in a generalized symplectic coset $\operatorname{GSp}^q_{2g}(\Z_p)$ 
($\operatorname{GSp}^q_{2g}(\Z_p)$ is the coset of matrices $M$ such that $M^tJM=qJ$, where $J$ is an invertible alternating matrix, and in particular the $M\in \operatorname{GSp}^q_{2g}(\Z_p)$ are invertible).  Friedman and Washington prove that the cokernels of random matrices $I-M$, where $M$ is random from the (multiplicative) Haar measure on $\GL_{2g}(\Z_p)$, as $g\ra\infty$, approach the Cohen-Lenstra distribution \cite[Section 4]{Friedman1989}.  
Eventually, it was understood that this also holds for $I-M$, where $M$ is random from the Haar induced measure on $\operatorname{GSp}^q_{2g}(\Z_p)$ and $gcd(q-1,p)=1$.  (This is not clearly stated in the literature, but follows from work of 
Achter \cite{Achter2008} and Ellenberg and Venkatesh \cite{Ellenberg2016} in a very round about way, as outlined by Garton \cite[p.153]{Garton2015}. )

We can view these results as additional examples of random matrices in the universality class of Theorem~\ref{T:NWpadic}, even though the matrices do not have independent entries, and also come from very special distributions.  Another example that would fit into this category is M\'{e}sz\'{a}ros's theorem \cite[Theorem 1]{Meszaros2020} that says that the Laplacians of uniform random $d$-regular directed graphs, for any $d\geq 3$, also have these limiting cokernel distributions.  
It is a very interesting problem to extend this universality to matrices with dependent entries but for broader classes of random matrices, where the degrees of freedom in choosing the distribution of random matrices is large.   As an example, in \cite[Theorem 1.6]{Nguyen2021}, we extend universality to Laplacians of random matrices with independent entries  (so matrices whose off-diagonal entries are independent and whose columns sum to $0$), which includes Laplacians of directed Erd\H{o}s-R\'enyi random graphs.  However, this is a very special kind of dependency among entries for which the methods are well-suited.

\begin{oproblem}
Extend Theorem~\ref{T:NWpadic} to more classes of matrices with dependent entries.  
\end{oproblem}

\begin{oproblem}
Give a unified proof that multiple special classes of random matrices are in the universality class of Theorem~\ref{T:NWpadic}.  
\end{oproblem}

\subsection{Random finite abelian groups with additional structure}

Of course, if the entries of the random matrices have too much dependence in some particular way, their cokernels may land in another universality class.  
For example, for symmetric matrices the author has proved the following.

\begin{theorem}[{\cite{Wood2017}}]\label{W:symmat}
Let $p$ be a prime and $0< \alpha<1$.
For every positive integer $n$,
let $M_n$  be a symmetric random $n\times n$ matrix  with independent entries $\xi_{i,j,n}\in \Z_p$ for $i\geq j$ that satisfy
$$
\max_{a\in \F_p} \Prob(\xi_{i,j,n}\equiv a \pmod{p}) \leq 1-\alpha
$$ for every $i,j,n$.  
Then for every abelian $P$-group $A$, we have
\begin{align*}
\lim_{n\ra\infty} \P(\cok(M_n)_P\isom A) 
&= \frac{\#\{\textrm{symmetric, bilinear, perfect }\phi:A\times A \ra \bC^*\}}{|A||\Aut(A)|}
\prod_{k=0}^\infty (1-p^{-2k-1}).
\end{align*}
\end{theorem}
(Note the number of pairings can be described explicitly in terms of the partition corresponding to the group $A$ \cite[{Equation (2)}]{Wood2017}.)

\begin{proof}
Theorem 6.1 in \cite{Wood2017} gives the moments, and then Theorem~\ref{T:Wmom} shows they determine a unique distribution, and \cite[Theorem 2]{Clancy2015} gives formulas for the 
distribution when $M_n$ is taken from Haar measure, as in \cite[Corollary 9.2]{Wood2017}.
\end{proof}

Nguyen and the author have forthcoming work in which we extend Theorem~\ref{W:symmat} to integer matrices (and all primes), analogous our results on $n\times n$ matrices over $\Z$ described above (including obtaining the probability that the cokernel is cyclic).   

One way of understanding why some random groups are in a different universality class is that the groups may be naturally coming with further structure than just group structure.  For example, the cokernel of a symmetric matrix over the integers (or $\Z_p$) \cite[Section 1.1]{Clancy2015} comes with a natural symmetric bilinear pairing.  Clancy, Leake, and Payne \cite{Clancy2015a} suggested that for random graphs, the cokernels of the graph Laplacian, along with their symmetric pairing, should be 
distributed proportionally to $|A|^{-1}|\Aut(A,\langle,\rangle)|^{-1}$.  If we sum these expressions over isomorphism classes of pairings for a fixed group, we exactly obtain the probabilities for groups in Theorem~\ref{W:symmat} (see \cite[Corollary 9.2]{Wood2017}).  This reflects an important part of the philosophy of the Cohen-Lenstra-Martinet heuristics---that the natural distributions on algebraic objects must take into account all of the structure of the objects.  For example, when considering class groups of Galois number fields with Galois group $G$, we consider the class group not just as a group but rather as a $G$-module, and the predicted probabilities for a particular $G$-module involve the number of automorphisms of the $G$-module (as a $G$-module).  Since the distributions that arise from universality theorems are certainly natural, we would expect them to share this sensitivity to extra structure, and thus it makes sense that cokernels of symmetric matrices, since as such they have natural symmetric pairings, should be distributed in a distribution that sees those pairings.

\begin{oproblem}
Prove that the cokernels of random symmetric matrices as in Theorem~\ref{W:symmat}, along with their pairings, are distributed as suggested by Clancy, Leake, and Payne \cite[Section 4]{Clancy2015a}.  One might naturally use moments of groups with pairings, and the corresponding moment problem, as in \cite[Section 8]{Lipnowski2020}.
\end{oproblem}

There are a few other classes of random groups that we know in this universality class.
The result \cite[Theorem 1.1]{Wood2017} extends Theorem~\ref{W:symmat} to cokernels of Erd\H{o}s--R\'{e}nyi random graph Laplacians, also known as sandpile groups or Jacobians of the graphs.  M\'esz\'aros  \cite[Theorem 1.2]{Meszaros2020} extends Theorem~\ref{W:symmat} to sandpile groups of $d$-regular graphs for $d\geq 3$ (unless $d$ is even and $p=2$, in which case a different distribution arises, likely reflecting further structure of the pairing).  Dunfield and Thurston \cite[Section 8.7]{Dunfield2006} show that the homology
$H_1(M,\F_p)$ for a 3-manifold from a random Heegaard splitting of genus $g$ as $g\ra\infty$ approaches the universal distribution of Theorem~\ref{W:symmat},
 or more precisely the pushforward of that distribution to elementary abelian $p$-groups under the map $A\mapsto A/pA$.  (See \cite[Corollary 9.4]{Wood2017} to see that this is indeed the pushforward.)  Forthcoming work of Sawin and the author finds the distribution more generally of 
 $H_1(M,\Z_p)$, along with the torsion linking pairing, of these random $3$-manifolds, and finds that it is in the natural distribution suggested by 
 Clancy, Leake, and Payne \cite[Section 4]{Clancy2015a}.  So the presence of the symmetric pairing from the torsion linking pairing explains why the homology of random $3$-manifolds appears in this universality class. 
 
\begin{oproblem}
Prove that the sandpile groups of Erd\H{o}s--R\'{e}nyi random graphs (or uniform random $d$-regular graphs) along with their pairings, are distributed as suggested by Clancy, Leake, and Payne \cite[Section 4]{Clancy2015a}.  
\end{oproblem} 
 

There are however, many more algebraic structures that are important in arithmetic statistics and other fields whose universality classes should be studied, such as random abelian groups with an action of a group, or random modules.
\begin{oproblem}\label{P:Gmods}
Prove an analog of Theorem~\ref{T:NWpadic} for $\Z_p[G]$-modules for a finite group $G$ (with $p\nmid |G|$).  
More generally (as would be related to the Cohen-Lenstra-Martinet heuristics for non-Galois fields, see \cite[Sections 7-8]{Wang2021}), prove an analog of Theorem~\ref{T:NWpadic} for random $\mathfrak{o}$-modules, where $\mathfrak{o}$ is a maximal order (over $\Z_p$) in a semisimple $\Q_p$-algebra.
\end{oproblem}

Note that the reduction of Problem~\ref{P:Gmods} mod $p$ is a question about matrices over finite fields.  So part of solving the above will include generalizing Theorem~\ref{T:NWmodp} from $\F_p$ to general finite fields $\F_q$.   In this more general case, the requirement that the $\xi_{i,j,n}$ are not concentrated at a single point is not sufficient, and must be replaced with something like  $\xi_{i,j,n}$ not concentrated on a translate of a subfield.  Kahn and Koml\'os \cite{Kahn2001} have shown universality of the singularity probability of a random $n\times n$ matrix over $\F_q$ under such a condition.

\begin{oproblem}
For $\mathfrak{o}$ a maximal order (over $\Z_p$) in a semisimple $\Q_p$-algebra, with an order two automorphism $\sigma$, such as $\mathfrak{o}$ being the ring of integers of the unramified quadratic extension of $\Q_p$, or $\mathfrak{o}=\Z_p \times \Z_p$, prove an analog of Theorem~\ref{W:symmat} for random $\sigma$-Hermitian matrices (i.e. $M$ such that $\sigma(M)=M^t$).
\end{oproblem}


\subsection{Random non-abelian groups}

We now turn to universality questions for non-abelian random groups, which are largely unstudied, but we expect contain much potential.  
One naturally starts with a free group (or free profinite group) $F_n$ and takes the quotient by independent random relations in some way that involves many independent choices for each relations.  As $n\ra\infty$, one hopes that the limiting distribution is somewhat insensitive to the distribution from which the relations are chosen. 
The first stumbling block when considering such questions is that it is less clear how to take a random relation built up from many independent choices.  When the relation was in $\F_p^n$ or $\Z_p^n$, we could just take each coordinate independently.  However, if $F_n$ is the free group (or free profinite group) on $n$ generators, there are not analogous coordinates in $F_n$.  In the case of random nilpotent groups, one might consider using Mal'cev coordinates.
Another way to characterize the probability measures on $\Z_p^n$ from which we drew relations above, e.g. in in Theorem~\ref{T:NWpadic}, is that they are not concentrated at a point in any finite simple quotient, so it may be interesting to consider the non-abelian version of that condition.  
While it is not so clear what the parameters for the universality class should be, one has a natural target for the universal distribution from a result of Liu and the author on the quotient of the free group by random relations.  Let $\mathcal{G}$ be the set of isomorphism classes of profinite groups with finitely many surjections to any finite group.
(This is $\mathcal{G}$ from Section~\ref{SS:nonabmom} with $\Gamma=1$, and we consider the same topology on it as defined there.)

\begin{theorem}[{\cite[Theorem 1.1]{Liu2020}}]\label{T:LW}
For every integer $u$, there is a random group $X_u$ in $\mathcal{G}$ whose measure is described explicitly on each basic open \cite[Equation (3.2)]{Liu2020}.
If $F_n$ is the free profinite group on $n$ generators, and $r_i$ are independent random elements of $F_n$ drawn from Haar measure, then
as $n\ra\infty$ the quotient 
$$F_n/\langle r_1,\dots,r_{n+u}\rangle \textrm{  weakly converge in distribution to } X_u.$$
\end{theorem}

As in \cite[Section 14]{Liu2020}, one can consider usual (not profinite) free group $F_n$ and take random relations obtained from a random walk on $F_n$.  However, as the length of the random walk goes to infinity, these relation become equidistributed with respect to Haar measure, and so this is not really a new example for the universality class.  

\begin{oproblem}
Find some more general hypotheses for a distribution on $F_n$ from which one can draw independent relations so that Theorem~\ref{T:LW} still holds.
\end{oproblem}

While it would be nice to have hypotheses that allow a wide range of distributions, i.e. a universality theorem, it would even be interesting  to find other specific random groups converging to the distributions $X_u$.    We give one example here, which is a non-abelian analog of the result of Friedman and Washington on cokernels of $I-M$, where $M$ is random from the Haar measure on $\GL_n(\Z_p)$.  

\begin{theorem}\label{T:Aut}
Let $F_n$ be the free profinite group on $n$ generators, and let $Aut(F_n)$ be the group of (continuous) automorphisms of $F_n$, which is a profinite group
\cite[Corollary 4.4.4]{Ribes2010}.  Let $I\in Aut(F_n)$ be the identity and let $\alpha_n$ be a random element of $Aut(F_n)$ with respect to Haar measure.
Then, as $n\ra \infty$
$$F_n/\langle \alpha_n(x)x^{-1} | x\in F_n\rangle \textrm{  weakly converge in distribution to } X_0,$$
(where $X_0$ is defined as in Theorem~\ref{T:LW}).
\end{theorem}

We will compute the moments of these random groups, and then apply Corollary~\ref{C:Sawin} from Sawin's result on the moment problem.  To do that, we first need the moments of $X_0$.  While it is an easy to see that for independent Haar relations $r_i$, we have 
$$\lim_{n\ra\infty} \E(\#\Sur(F_n/\langle r_1,\dots,r_{n+u}\rangle,A))=1,$$ it does require some argument to interchange the limit in $n$ and the expectation and obtain these same moments for $X_0$.  

\begin{lemma}\label{L:momX}
Let $X_0$ be defined as in Theorem~\ref{T:LW}.  Then for any finite group $H$, we have
$$
\E(\Sur(X_0,H))=1.
$$
\end{lemma}
\begin{proof}
We follow the strategy of \cite[Theorem 6.2]{Liu2019} adapted to our situation.  Let $F_n$ be the free profinite group on $n$ generators and let $Z_n$ be the random profinite group $F_n/\langle r_1,\dots,r_{n}\rangle$, where the $r_i$ are random elements of $F_n$ from Haar measure.
For any positive integer $\ell$, let $\calC_\ell$ be the set of finite groups of order at most $\ell$.
We consider the following function defined for any positive integer $\ell$ and any finite group $G$ of level with $G^{\calC_\ell}\isom G$, 
	\[
		f_n(G, \ell) = \E( |\Sur(Z_n, H)| \times \mathbbm{1}_{{Z_n}^{\calC_\ell} \simeq G}),
	\]
	where $\mathbbm{1}_{{Z_n}^{\calC_\ell} \simeq G}$ is the indicator function of
 ${Z_n}^{\calC_\ell} \simeq G$. 
 We let $\pi_{\calF_n}\colon  \calF_n \to (\calF_n)^{\calC_\ell}$ and $\pi_H: H \to H^{\calC_\ell}$ be the natural quotient maps. Each $\phi\in \Sur(\calF_n, H)$ induces a map $\overline{\phi} \in \Sur((\calF_n)^{\calC_\ell}, H^{\calC_\ell})$.
 By the definition of random group $Z_{n}$, we have 
	\begin{eqnarray}
		&&\E( |\Sur(Z_n, H)| \times \mathbbm{1}_{{Z_n}^{\calC_\ell} \simeq G})\label{eq:long-1}\\
		&=& \sum_{\phi \in \Sur(\calF_n, H)} \Prob( r_1,\dots,r_{n} \in \ker \phi \text{ and } (\calF_n)^{\calC_\ell}/
		\langle \pi_{\calF_n}(r_1),\dots,\pi_{\calF_n}(r_{n}) \rangle
		 \simeq G).\nonumber 
	\end{eqnarray}
 Given $\phi\in \Sur_{\Gamma}(\calF_n, H)$, and $y_1,\dots,y_n\in \ker \overline{\phi}$,
 we have that
		\[
		\Prob(r_1,\dots,r_{n} \in \ker \phi  \mid 
\pi_{\calF_n}(r_i)=y_i \textrm{ for all $i$})= \frac{|H^{\calC_\ell}|^{n}}{|H|^{n}}.
	\]
This follows from the straightforward calculation that $|\pi_{F_n}^{-1}(y_i) \cap \ker \phi|=|F_n||H^{\calC_\ell}|/(|F_n^{\calC_\ell}||H|)$
Then, summing over choices of $y_i\in \ker\overline{\phi}$ such that $(\calF_n)^{\calC_\ell}/		\langle y_1,\dots,y_{n} \rangle \isom G$, we have
\begin{align*}
&\Prob\left(
\begin{aligned}
&r_1,\dots,r_{n} \in \ker \phi \text{ and } \\
&(\calF_n)^{\calC_\ell}/		\langle \pi_{\calF_n}(r_1),\dots,\pi_{\calF_n}(r_n) \rangle \isom G
\end{aligned}
 \,\Bigg|\,
  \begin{aligned}
&\pi_{\calF_n}(r_i) \in \ker \overline{\phi} \text{ for all $i$, and } \\
&(\calF_n)^{\calC_\ell}/		\langle \pi_{\calF_n}(r_1),\dots,\pi_{\calF_n}(r_n) \rangle \isom G
\end{aligned}
\right)
=&\frac{|H^{\calC_\ell}|^{n}}{|H|^{n}}.
\end{align*}
Thus
	\eqref{eq:long-1} is equal to 
	\begin{eqnarray}
		&& \frac{|H^{\calC_\ell}|^{n}}{|H|^{n}} \sum_{\phi \in \Sur(\calF_n, H)} \Prob\left( \begin{aligned}
&\pi_{\calF_n}(r_i) \in \ker \overline{\phi} \text{ for all $i$, and } \\
&(\calF_n)^{\calC_\ell}/		\langle \pi_{\calF_n}(r_1),\dots,\pi_{\calF_n}(r_n) \rangle \isom G
\end{aligned}\right) \nonumber \\
		&=& \frac{|H^{\calC_\ell}|^{n}}{|H|^{n}}  \sum_{\phi \in \Sur(\calF_n, H)} \frac{\#\left\{(\tau, \pi) \,\Bigg|\, \begin{aligned} 
		&\tau\in \Sur((\calF_n)^{\calC_\ell}, G)\\
		&\pi \in \Sur(G, H^{\calC_\ell}) \\
		& \text{ and } \pi \circ \tau = \overline{\phi}
		\end{aligned}\right\}}{|\Aut(G)| |G|^{n}} P_{0,n}(U_{\calC_\ell, G}), \label{eq:long-2}
	\end{eqnarray}
where $P_{0,n}(U_{\calC_\ell, G})$ is defined in \cite{Liu2020} just before Lemma 9.5, and is the probability that $n$ independent uniform random elements in the kernel of $(\calF_n)^{\calC_\ell} \ra G$ generate that kernel as a normal subgroup (as worked out in the proof of \cite[Theorem 8.1]{Liu2020}). 
(To explain the above equality a bit more: if $(\calF_n)^{\calC_\ell}/		\langle \pi_{\calF_n}(r_1),\dots,\pi_{\calF_n}(r_n) \rangle \isom G$
then there is a choice of $\tau\in \Sur((\calF_n)^{\calC_\ell}, G)$ inducing that isomorphism, whose $\Aut(G)$ orbit is unique, and $\bar{\phi}$ must factor through $\tau$ since $\pi_{\calF_n}(r_i) \in \ker \overline{\phi}$.  Given a $\tau$, the probability that the relations are in $\ker \tau$ and generate it as a normal subgroup is $|G|^{-n} P_{0,n}(U_{\calC_\ell, G})$.)
 	
	On the other hand, let $\overline\phi \in \Sur((\calF_n)^{\calC_\ell}, H^{\calC_\ell})$. Then the composition map $\rho:= \overline{\phi} \circ \pi_{\calF_n}$ is a  surjection  $\calF_n\to H^{\calC_\ell}$. The number of $\phi\in \Sur(\calF_n, A)$ such that $\phi$ induces  $\overline{\phi}$ we denote by $\Sur(\rho, \pi_H)$. It is easy to see that
		\[
			\frac{|\Sur((\calF_n)^{\calC_\ell}, G)|}{|G|^n}\leq1\quad \textrm{and} \quad 
			\lim_{n\to \infty} \frac{|\Sur((\calF_n)^{\calC_\ell}, G)|}{|G|^n}=1,
		\] and similarly
		\[
\frac{|\Sur(\rho, \pi_H)|}{|H|^n |H^{\calC_\ell}|^{-n}}\leq 1, \quad \lim_{n \to \infty} \frac{|\Sur(\rho, \pi_H)|}{|H|^n |H^{\calC_\ell}|^{-n}}=1.
		\]
		Then by \eqref{eq:long-2}, we obtain that $f_n(G, \ell)=g_n(G, \ell) P_{0,n}(U_{\calC_\ell,G})$ where
		\[
		g_n(G, \ell)=\frac{|H^{\calC_\ell}|^{n} |\Sur(\rho, \pi_H)| |\Sur((\calF_n)^{\calC_\ell},G)| |\Sur(G, H^{\calC_\ell})|}{|H|^{n} |G|^{n} |\Aut(G)|}
		\]
		\[
			\text{and}\quad g(G,\ell):=\lim_{n\to \infty} g_n(G, \ell) =\frac{|\Sur(G, H^{\calC_\ell})|}{  |\Aut(G)|}.
		\]
		Now we apply \cite[Lemma 5.10]{Liu2019}, where condition (1) holds by definition, (2) from the above, and (3) follows from the definition of $f_n(G,\ell)$.
This allows us to conclude, for every $\ell$
		\begin{equation}\label{eq:long-3}
			\sum_{\substack{\text{$G$}\\ G^{\calC_\ell}\isom G  }} \lim_{n\to\infty} f_n(G, \ell) =\lim_{n\to \infty} f_n(\text{trivial group}, 1) = \lim_{n\to \infty} \E(|\Sur(Z_{n}, H)|) = 1.
		\end{equation}
		When $\ell$ is sufficiently large such that $H^{\calC_\ell}\isom H$, 
		\[
			\lim_{n\to \infty} f_n(G, \ell) = \lim_{n\to \infty} |\Sur(G, H)| \Prob((Z_{n})^{\calC_\ell} \simeq G) = |\Sur(G, H)| \Prob((X_0)^{\calC_\ell} \simeq G),
		\]
where the last equality is by Theorem~\ref{T:LW}.  Hence \eqref{eq:long-3} gives the desired result in the lemma.
\end{proof}

\begin{proof}[Proof of Theorem~\ref{T:Aut}]
We compute the moments of $F_n/\langle \alpha_n(x)x^{-1} | x\in F_n\rangle$.
Consider a fixed finite group $H$.  If $H$ can be generated by $n$ elements, then there are some number of surjections $\phi: F_n \ra H$.
Those surjections that factor through the quotient $F_n/\langle \alpha_n(x)x^{-1} | x\in F_n\rangle$ are exactly those $\phi$ such that $\phi\alpha=\phi$.
So
$$
\E(\#\Sur(F_n/\langle \alpha_n(x)x^{-1} | x\in F_n\rangle,H ))=\sum_{\phi\in \Sur(F_n,H)} \Prob(\phi\alpha=\alpha).
$$
The action of $\Aut(F_n)$ on $\Sur(F_n,H)$ is transitive \cite[Proposition 2.2]{Lubotzky2001}, and factors through a finite group.
So $\Prob(\phi\alpha=\alpha)=|\Sur(F_n,H)|^{-1}$.  Thus, as long as $H$ can be generated by $n$ elements, we have
$$
\E(\#\Sur(F_n/\langle \alpha_n(x)x^{-1} | x\in F_n\rangle,H ))=1.
$$
Thus we can use Theorem~\ref{T:Sawin} and Lemma~\ref{L:momX} to conclude the theorem.
\end{proof}

Of course, if any kind of universality result can be proven for non-abelian random groups, it would then be interesting to extend the methods to particular non-abelian groups with additional structure that are arising in number theory and topology.  
So far the applications of these sort of universality methods for random groups have largely been in combinatorics.  We expect that as the methods become developed, there will be further applications, including in number theory and topology.


\newcommand{\etalchar}[1]{$^{#1}$}
\def\cprime{$'$}

\end{document}